\documentclass[11pt,reqno]{amsart}
\usepackage{amsmath,amssymb,amsthm,mathtools,color}
\usepackage{graphicx}
\usepackage[shortlabels]{enumitem}
\usepackage{hyperref}

\usepackage[margin=1.20in]{geometry}
\numberwithin{equation}{section}
\newtheorem{theorem}{Theorem}[section]
\newtheorem{proposition}[theorem]{Proposition}
\newtheorem{lemma}[theorem]{Lemma}

\theoremstyle{definition}
\newtheorem{definition}[theorem]{Definition}
\newtheorem{remark}[theorem]{Remark}

\newcommand{\R}{\mathbb{R}}
\newcommand{\C}{\mathbb{C}}
\newcommand{\Z}{\mathbb{Z}}
\newcommand{\N}{\mathbb{N}}

\newcommand{\K}{\mathcal{K}}
\newcommand{\Ell}{\mathcal{E}}
\newcommand{\eps}{\varepsilon}
\newcommand{\ov}{\overline}

\title[Finite-time blow-up for the defocusing energy-supercritical NLW]{Finite-time blow-up for the real-valued defocusing\\ energy-supercritical NLW}

\author{M. Cadiot}
\address{\'Ecole Polytechnique}
\email{matthieu.cadiot@polytechnique.edu}
\author{A. D. Ionescu}
\address{Princeton University}
\email{aionescu@math.princeton.edu}
\author{S. Palasek}
\address{Princeton University}
\email{spalasek@princeton.edu}


\begin{document}

\begin{abstract}
We construct real-valued solutions of the defocusing energy-supercritical nonlinear wave equation $\partial_t^2u-\Delta u+u^p=0$ that blow up in finite time from smooth, compactly supported initial data. We consider three cases: for $p=3$ we construct blow-up solutions in all dimensions $d\geq 9$, for $p=5$ we treat all dimensions $d\geq 7$, while for $p=7$ we treat all dimensions $d\geq 6$. 

The solutions we construct here appear to be the first blow-up solutions of defocusing energy-supercritical nonlinear wave equations with smooth real-valued data.  The main step is the construction of nonzero \emph{discretely self-similar profiles} that are periodic in self-similar time and real-analytic across the light cone. We find approximate profiles by numerical continuation from rotating self-similar solutions of the complex-valued equation and upgrade them to exact solutions by a computer-assisted Newton--Kantorovich argument in a Banach algebra of Fourier--Chebyshev coefficients. The main analytic ingredients are an exact description of the linear mode operators and uniform bounds for their inverses.

The codes for the computer-assisted parts of the proofs, together with the approximate solutions, are available on GitHub \cite{code}.
\end{abstract}

\maketitle
\tableofcontents

\section{Introduction}\label{sec:intro}

\subsection{The problem}
We consider the defocusing nonlinear wave equations
\begin{equation}\label{eq:nlw}
\partial_t^2u-\Delta u+u^p=0,\qquad u:I\times\R^{d}\to\R,
\end{equation}
where $I\subset\R$ is an interval, $d\geq 2$ and $p\geq 3$ is odd. The equation is invariant under the scaling $u_\lambda(t,x):=\lambda^{\frac{2}{p-1}}u(\lambda t,\lambda x)$, which preserves the norm of $\dot H^{s_c}\times\dot H^{s_c-1}$, where $s_c:=\frac{d}{2}-\frac{2}{p-1}$ is the critical regularity. Regular solutions conserve the energy
\[
E(u):=\int_{\R^{d}}\tfrac12|\partial_tu|^2+\tfrac12|\nabla u|^2+\frac{1}{p+1}u^{p+1}\,dx.
\]

For energy-subcritical ($s_c<1$) defocusing wave equations, smooth solutions are global by classical arguments (see e.g.\ \cite{SSbook,So}); in the energy-critical case ($s_c=1$), global regularity was established by Struwe \cite{St}, Grillakis \cite{Gr} and Shatah--Struwe \cite{SS1,SS2}. In the energy-supercritical case ($s_c>1$), global regularity of smooth solutions is a well-known open problem. On the positive side, conditional results show that bounded critical norms imply global existence and scattering in various ranges of dimensions and powers; see for example Kenig--Merle \cite{KM}, Killip--Visan \cite{KV1,KV2}, Bulut \cite{Bu}, Dodson--Lawrie \cite{DL} and Duyckaerts--Yang \cite{DY}. Large global solutions were constructed by Krieger--Schlag \cite{KS}. 

Concerning singularity formation, Tao \cite{Ta,Ta2} constructed finite-time blow-up solutions for defocusing supercritical wave and Schr\"odinger systems with specially designed nonlinearities. Merle--Rapha\"el--Rodnianski--Szeftel \cite{MRRS1,MRRS1b,MRRS2} discovered the front compression mechanism, through the construction of smooth self-similar imploding profiles for the compressible Euler equations, and used it to prove blow-up for defocusing energy-supercritical nonlinear Schr\"odinger equations in several cases. The construction was extended by Buckmaster--Cao-Labora--G\'omez-Serrano \cite{BCG} to other adiabatic exponents $\gamma>1$,  by Cao-Labora--G\'omez-Serrano--Shi--Staffilani \cite{CGSS1,CGSS2} to non-radial imploding solutions in the whole space and on the torus, and by Buck--Carballeira--G\'omez-Serrano--Shi \cite{BCGS}, with computer assistance, to defocusing energy-supercritical NLS in low dimensions.

The same mechanism was adapted to prove finite-time blow-up for \emph{complex-valued} solutions of defocusing supercritical wave equations, by Shao--Wei--Zhang \cite{SWZ} and Buckmaster--Chen \cite{BC}. For real-valued solutions of defocusing wave equations the question has remained open; see the discussion in \cite{BC,SWZ}. 

In this paper we solve this problem in the cases $(p=3, d\geq 9)$, $(p=5, d\geq 7)$, and $(p=7, d\geq 6)$, by proving finite-time blow-up of smooth real-valued solutions of \eqref{eq:nlw}. Precisely, we construct \emph{exact discretely self-similar} radial solutions in a backward light cone in the cases $(p=3, d=9)$, $(p=5, d=7)$, and $(p=7, d=6)$, with profiles periodic in self-similar time and real-analytic across the cone. Adding dummy variables in higher dimensions, cutting off their initial data, and using finite speed of propagation leads to blow-up solutions that remain smooth until the blow-up time.

We stress that the main difficulty concerns smooth solutions. Rough self-similar solutions of the defocusing equation, which are not smooth across the light cone, can be constructed by perturbative methods, as noted in \cite[footnote~2]{Ta} with reference to \cite{Pl,RY}. In Appendix~\ref{app:rough} we give a self-contained version of this remark, valid for every energy-supercritical exponent, with bounded, compactly supported, finite-energy data. 

Self-similar blow-up is classical for focusing wave equations, where the profiles are stationary in self-similar variables; see for instance \cite{MZ,BMW,BBMW,DS,DD,GlS,CGS} and the references therein. Discretely self-similar solutions, whose profiles are periodic in the self-similar time, were first observed by Choptuik \cite{Ch} in the critical collapse of a self-gravitating scalar field, and were further investigated by Reiterer--Trubowitz \cite{RT} (with computer assistance) and by Cicortas--Kehle \cite{CK}. Discretely self-similar blow-up for a geometric wave equation, with values in $\mathbb S^1$, was recently constructed by Glogi\'c--Hilditch--Wallauch \cite{GHW}. 

We prove the existence of the profiles by a computer-assisted Newton--Kantorovich argument around numerical approximations, using Fourier--Chebyshev expansions and rigorous error bounds. This approach, often formulated in terms of radii polynomials, goes back to \cite{DLM}; see \cite{LR} for Chebyshev series, \cite{BBLV} for time-periodic solutions of a PDE, and \cite{vdBL,GS} for surveys. It has developed into a flexible tool for the constructive existence of specific solutions of nonlinear PDEs, for instance on unbounded domains \cite{CLN}, for the rigorous integration of parabolic equations with fully spectral Fourier--Chebyshev expansions \cite{CL, CL22, vdBBS, TLJO, DLT}, and recently for overhanging periodic steady water waves \cite{CaHa}. In the study of singularity formation, computer-assisted proofs of this and related types appear in \cite{RT,CH1,CH2,BCG,DF,BCGS}.

\subsection{The main theorem}
We consider the three cases
\begin{equation}\label{eq:cases}
(d,p)\in\big\{(9,3),\ (7,5),\ (6,7)\big\},\qquad\text{for which}\qquad s_c=\tfrac72,\ 3,\ \tfrac83\ \text{ respectively,}
\end{equation}
and we write $\alpha:=\frac2{p-1}$. For $T>0$ let
\[
\K_T:=\{(t,x)\in[0,T)\times\R^{d}:|x|\le T-t\}
\]
denote the backward light cone with vertex $(T,0)$.

\begin{theorem}\label{thm:main}
Let $(d,p)$ be one of the pairs in \eqref{eq:cases}.

(i) There exist $\omega^*>0$ and a real-analytic function $W^*:\R\times(-\rho_*,\rho_*)\to\R$, where $\rho_*:=1.0011$, such that $W^*$ is not identically zero, is $2\pi$-periodic in its first variable and even in its second variable, and satisfies $W^*(\theta+\pi,\rho)=-W^*(\theta,\rho)$. Moreover
\begin{equation}\label{eq:omega}
|\omega^*-\omega_{d,p}|\le 10^{-10},\qquad \omega_{d,p}:=\begin{cases}
40.2602362339 & \text{if } (d,p)=(9,3),\\
24.0682155676 & \text{if } (d,p)=(7,5),\\
19.9084444867 & \text{if } (d,p)=(6,7),
\end{cases}
\end{equation}
and for every $T>0$ the function
\begin{equation}\label{eq:ustar}
u^*(t,x):=(T-t)^{-\alpha}\,W^*\Big(\omega^*\log\frac1{T-t},\,\frac{|x|}{T-t}\Big)
\end{equation}
is a smooth solution of \eqref{eq:nlw} on $\K_T$.

(ii) For every $T>0$ there are radial data $(u_0,u_1)\in C_c^\infty(\R^{d})\times C_c^\infty(\R^{d})$ such that the maximal smooth solution $u$ of \eqref{eq:nlw} with $(u,\partial_tu)(0)=(u_0,u_1)$ coincides with $u^*$ on $\K_T$, and extends to a $C^\infty$ function on $([0,T]\times\R^{d})\setminus\{(T,0)\}$. In particular, $u$ is bounded on $[0,T)\times\{|x|\ge\eps\}$ for every $\eps>0$, while
\[
\|u(t)\|_{L^\infty(\R^{d})}\ge(T-t)^{-\alpha}\qquad(0\le t<T),
\]
so that $\lim_{t\uparrow T}\|u(t)\|_{L^\infty(\R^{d})}=\infty$. The scale-invariant norms of $u(t)$ on $\{|x|\le T-t\}$, $\|(\nabla u,\partial_tu)(t)\|_{L^{d(p-1)/(p+1)}}$ and $\|u(t)\|_{L^{d(p-1)/2}}$, are continuous periodic functions of $\log(T-t)$, and the energy of $u(t)$ on $\{|x|\le T-t\}$ is bounded by $C(T-t)^{2(s_c-1)}$.

(iii) Moreover, for every $d'>d$, writing $x=(x',x'')\in\R^{d}\times\R^{d'-d}$ and taking $\chi\in C_c^\infty(\R^{d'-d})$ with $\chi=1$ on $\{|x''|\le2T\}$, the maximal smooth solution $v$ of \eqref{eq:nlw} in $\R^{1+d'}$ with the smooth, compactly supported data $\big(\chi(x'')u_0(x'),\chi(x'')u_1(x')\big)$ has lifespan $T'\le T$, coincides with $u(t,x')$ on $\{|x''|\le2T-t\}$ for $t<T'$, and satisfies $\limsup_{t\uparrow T'}\|v(t)\|_{L^\infty(\R^{d'})}=\infty$.
\end{theorem}

\begin{figure}[t]
\centering
\IfFileExists{fig_profiles.pdf}{\includegraphics[width=\textwidth]{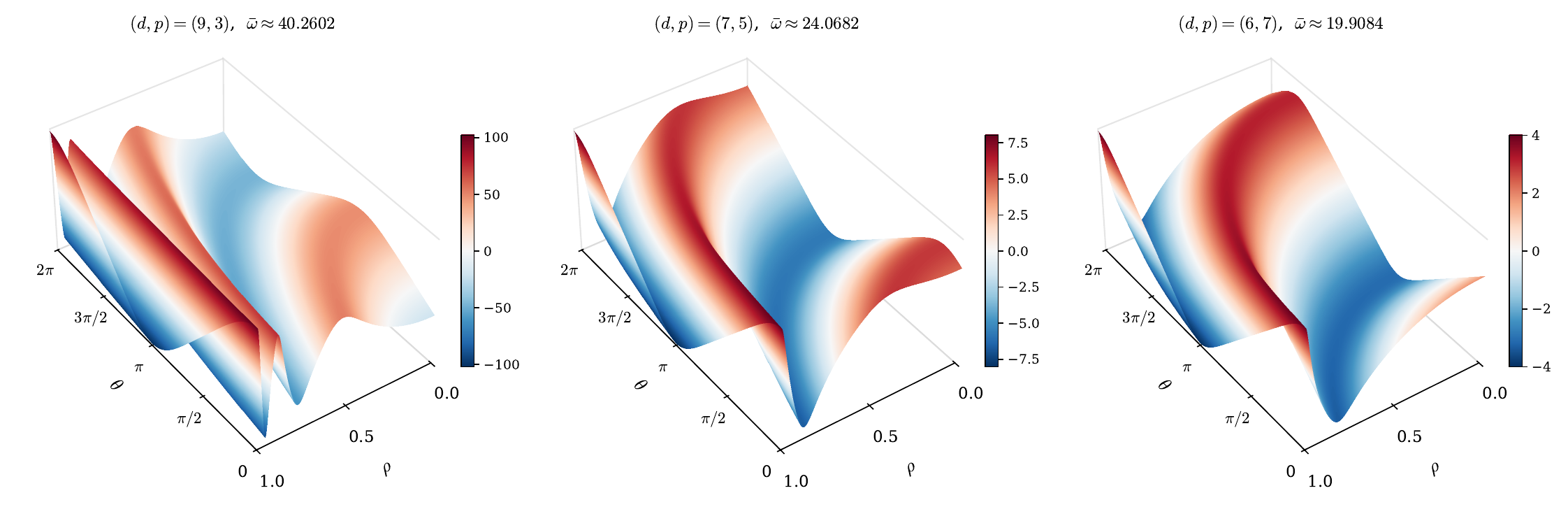}}{\fbox{\texttt{fig\_profiles.pdf}}}
\caption{The approximate profiles $\bar W(\theta,\rho)$ of Proposition~\ref{th:proof} on $[0,2\pi]\times[0,1]$, for $(d,p)=(9,3)$, $(7,5)$ and $(6,7)$ (the color scales differ). The profiles $W^*$ of Theorem~\ref{thm:main} satisfy $|W^*-\bar W|\le\delta_W$ on this set, with $\delta_W$ given in Table~\ref{tab:proof}.}
\label{fig:profile}
\end{figure}

\begin{remark}\label{rem:intro}
Several remarks on the theorem are in order.

\begin{enumerate}[(i)]
\item \emph{The profile equation.} The profiles $W^*$ solve the hyperbolic equation \eqref{eq:ssw} below, for which the light cone $\rho=1$ is characteristic. Smoothness across $\rho=1$ is essential; our profiles are real-analytic there. For each nonzero Fourier mode in the self-similar time, the linear equation admits a singular branch behaving like $(1-\rho)^{s_c-\frac12-in\omega}$, which smoothness excludes.

\item \emph{The computer-assisted part.} The proof uses the finite-dimensional inequalities of Proposition~\ref{th:proof}. They are verified by computer with rigorous error control: all the computations are done in interval arithmetic \cite{julia_interval,julia_olivier}, and the code is available at \cite{code}. 

\item \emph{Discretely self-similar profiles.} The profiles $W^*$ are genuinely time-dependent. We were not able to find smooth real self-similar (that is, independent of the self-similar time) profiles of \eqref{eq:nlw}, and for $s_c\le\frac32$ the maximum principle rules them out (cf.\ Remark~\ref{rem:noCSS}). Rough profiles, however, exist for every energy-supercritical exponent (Appendix~\ref{app:rough}).

\item \emph{Other exponents and open questions.} The continuation of Section~\ref{sec:profiles} in $d$ and $p$ also produces approximate real profiles for other exponents, for instance $(d,p)=(10,5)$, $(10,3)$, $(8,7)$, $(6,9)$, $(5,13)$, and, from other families of rotating profiles, for smaller dimensions such as $(8,3)$, $(6,5)$ and $(5,7)$. The analytic framework of Sections~\ref{sec:modes} and~\ref{sec:proof} applies to all these cases. The computer-assisted existence proof is carried out here only for the three pairs in \eqref{eq:cases}: for each exponent $p$, they have the smallest dimension among the profiles we found for which the finite-dimensional part of the proof remains manageable. It is not clear to us whether this mechanism, based on radial discrete self-similarity, could be applied to prove blow-up in the full supercritical range $s_c>1$, or whether this is possible only for sufficiently supercritical problems. The stability of the blow-up is also a natural question. We hope to return to these issues in the future.
\end{enumerate}
\end{remark}

\subsection{Strategy of the proof and the main ideas}\label{sec:novel} We now discuss some of the main steps in the proof of Theorem~\ref{thm:main}.

\subsubsection{Reduction to a fixed-point problem.} By standard arguments (see Section~\ref{sec:reduction}), it suffices to construct, for each of the three pairs in \eqref{eq:cases}, a real-valued discretely self-similar (DSS) profile $W$: a nonzero real-analytic solution of the self-similar equation \eqref{eq:ssw} on $\R\times(-\rho_1,\rho_1)$, $\rho_1>1$, which is $2\pi$-periodic in $\theta$, with the frequency $\omega$ as an unknown. We impose the half-period antisymmetry $W(\theta+\pi)=-W(\theta)$, compatible with the odd nonlinearity, and a phase normalization. 

For each profile we expand $$W(\theta,\rho)=\sum_{n\ \mathrm{odd}}V_n(\rho)e^{in\theta},\qquad V_n(\rho)=\sum_kW_{n,k}T_k(2\rho^2-1),$$ where $T_k$ is the Chebyshev polynomial of degree $k$, and work in a weighted $\ell^1$ space of coefficients, which is a Banach algebra and whose elements are analytic in a strip in $\theta$ and on a Bernstein ellipse in $2\rho^2-1$ containing a neighborhood of the light cone. In terms of the coefficients, the profile equation becomes $\Lambda(\omega)W+W^p=0$, where $\Lambda_n(\omega)$ is the linear operator acting on the $n$-th Fourier mode and $W^p$ is the $p$-th power of $W$ in the Banach algebra; a phase condition removes the invariance under translations in $\theta$. Existence is then proved by the Newton--Kantorovich theorem, applied to numerical approximate zeros $\bar U=(\bar\omega,\bar W)$, represented in Figure~\ref{fig:profile}, and an approximate inverse $A$ of the linearization, which coincides with $\Lambda(\bar\omega)^{-1}$ on the high modes. Thus, the main steps of the paper are the construction of good approximate DSS profiles $\bar U$ and the analysis of the linear operators $\Lambda_n(\omega)$, whose inverses must be controlled uniformly in $n$.

\subsubsection{The numerical profiles} Finding real-valued DSS profiles is a significant difficulty. Real self-similar smooth profiles, independent of the self-similar time, do not seem to exist (Remark~\ref{rem:noCSS}), and direct numerical searches for time-periodic profiles are challenging. We found the profiles by numerical continuation from the rotating self-similar solutions of the complex equation, $u=(T-t)^{-\alpha-i\omega}Q(|x|/(T-t))$, which are comparatively easy to compute, through a family of two-component systems that breaks the $U(1)$ symmetry. 

The computed branch reaches real profiles through apparent pitchfork bifurcations. Section~\ref{sec:origin} describes this numerical route and explains the transverse degeneracies predicted by symmetry. The continuation supplies the approximate profiles; their rigorous certification is carried out directly and does not require a proof of the continuation path or its apparent bifurcations. The same numerical continuation idea was used by Ionescu--Jia--Palasek \cite[Section~2.2]{IJP} to construct swirl-free, axially symmetric approximate profiles for the self-similar Navier--Stokes equations.

\subsubsection{The linear analysis} The Newton--Kantorovich argument requires an approximate inverse of the linearization, and only finitely many Fourier modes can be treated numerically. For all the other modes we must invert the linear operators $\Lambda_n(\omega)$ analytically, with bounds that are uniform in $n$. Two features of the problem make this delicate.

(1) At the characteristic cone $\rho=1$, the coefficient of the radial second derivative vanishes. For every nonzero Fourier mode, the linear equation has a regular branch and a singular branch behaving like $(1-\rho)^{s_c-\frac12-in\omega}$. We select the regular branch by requiring analyticity across $\rho=1$. This condition is built into the function space through the analytic weights, and no boundary condition is imposed at the cone. A similar device, a change of variables that builds parity conditions into the function space, is used in \cite{IJP}.

(2) The operators $\Lambda_n$ are non-normal operators of transport type, and their inverses can be much larger than the size of their diagonal entries suggests; the relevant quantity is a product of many transfer matrices, which passes through a turning point where two eigenvalues of modulus one collide.

The structure of the operators makes the analysis possible. In the variable $z=\rho^2$, the linear part of the mode-$n$ equation is a hypergeometric operator (Lemma~\ref{lem:hyp}). In the Chebyshev basis it is upper triangular. Consequently every column of $\Lambda_n^{-1}$ is computed exactly by a backward recursion with $2\times2$ transfer matrices (Lemma~\ref{lem:rec}). The singular branches are not analytic at $\rho=1$, while the elements of our spaces are, and $\Lambda_n^{-1}$ is bounded on these spaces (Lemma~\ref{lem:kappa}). In particular there are no small divisors, in contrast with time-periodic problems for wave equations on compact domains \cite{CW,Bo,Be}; this reflects the fact that no characteristic enters the backward light cone through its lateral boundary.

The key estimate is
\[
\|\Lambda_n^{-1}\|_{\ell^1_\nu\to\ell^1_\nu}\le C(n\omega)^{-2},\qquad |n|>K_1,
\]
with an explicit constant $C$ (cf. Lemma~\ref{lem:T1}). Its proof does not analyze the turning point. Instead, we use a single norm on $\C^2$, adapted to the eigenvalue $-1$ shared by all the limiting transfer matrices, in which these matrices have norm at most $\frac{1+\nu}2$ (Lemma~\ref{lem:transfer}); the weight $\nu>\frac{1+\nu}2$ of the norm, that is, the analyticity of the profile across the light cone, then dominates the products of transfer matrices, and the bound follows. The same argument controls the operators $\partial_\omega\Lambda_n\Lambda_n^{-1}$, which account for the dependence on the frequency $\omega$, and the far Chebyshev columns of $\Lambda_n^{-1}$ for the finitely many remaining modes (Lemma~\ref{lem:fixedn}). None of this depends on the particular choice of $(d,p)$; see Remark~\ref{rem:general}. These fundamental linear estimates allow us to derive explicit and quantitative bounds for the Newton--Kantorovich theorem (Proposition~\ref{th : radii polynomial}), leading to the constructive existence proofs of the blow-up profiles.

\subsection{Organization}
In Section~\ref{sec:reduction} we introduce self-similar variables, state the existence of the profiles (Proposition~\ref{prop:profiles}) and derive Theorem~\ref{thm:main} from it. Section~\ref{sec:origin} explains how the profiles were found, starting from rotating profiles of the complex equation; it is not used in the proof. In Section~\ref{sec:modes} we describe the mode operators in the Chebyshev basis, set up the zero-finding problem, recall the Newton--Kantorovich theorem, and construct the approximate solution and the approximate inverse. Section~\ref{sec:proof} contains the computation of the bounds required by the Newton--Kantorovich theorem, including the uniform bounds for the inverses of the mode operators (Section~\ref{sec:main}), and the proof of Proposition~\ref{prop:profiles} (Section~\ref{sec:result}), which completes the proof of Theorem~\ref{thm:main}. Appendix~\ref{app:rough} contains, for comparison, an elementary construction of rough self-similar blow-up solutions, for every energy-supercritical exponent.

\subsection{Acknowledgements and AI disclosure}

A.~D.~I.\ was supported in part by NSF grant DMS-2407694
and a Simons Collaboration Grant on Wave Turbulence.   M.~C.\ was supported by the ANR project CAPPS (ANR-23-CE40-0004-01) and by the FMJH (ANR-22-EXES-0013). The authors acknowledge that this work was performed using the Princeton Research Computing resources, a consortium including the Princeton Institute for Computational Science and Engineering (PICSciE) and Research Computing at Princeton University.

This paper and the code for the computer-assisted proof were written by the authors and were not generated by AI. However, Anthropic's Claude Fable 5.1 played a significant role in the discovery of the discretely self-similar solutions, by the numerical continuation method described in Section~\ref{sec:origin}, while both Claude Fable 5.1 and OpenAI's ChatGPT Astra 6 assisted with language editing, bibliographic checks, checking calculations, and the verification of the rigorous numerics. The authors reviewed all AI-assisted revisions and take full responsibility for the paper.

\section{Self-similar variables and the reduction to profiles}\label{sec:reduction}

\subsection{The profile equation} Assume that $d\in\{3,4,\ldots\}$ and $p\geq 3$ is an odd integer such that
\begin{equation}\label{eq:supercritical}
d>2+\frac{4}{p-1},\qquad\text{that is}\qquad s_c:=\frac d2-\alpha>1,\qquad\text{where}\quad\alpha:=\frac{2}{p-1}.
\end{equation}
Let $T>0$ and $\omega>0$, and introduce the self-similar variables
\begin{equation}\label{eq:ssvar}
\tau:=\log\frac1{T-t},\qquad\theta:=\omega\tau,\qquad\rho:=\frac{|x|}{T-t},\qquad u(t,x):=(T-t)^{-\alpha}\,W(\theta,\rho).
\end{equation}
The frequency $\omega$ is a free parameter, which normalizes the period: a function $W$ which is $2\pi$-periodic in $\theta$ corresponds to a solution which is $2\pi/\omega$-periodic in the self-similar time $\tau$.

\begin{lemma}\label{lem:ss}
A radial function $u$ of the form \eqref{eq:ssvar} solves \eqref{eq:nlw} in $\{|x|<T-t\}$ if and only if $W$ solves, for $\rho\in(0,1)$,
\begin{equation}\label{eq:ssw}
\begin{split}
\omega^2W_{\theta\theta}&+2\omega\rho W_{\theta\rho}+(2\alpha+1)\omega W_\theta-(1-\rho^2)W_{\rho\rho}\\
&-\Big(\frac{d-1}\rho-2(\alpha+1)\rho\Big)W_\rho+\alpha(\alpha+1)W+W^p=0.
\end{split}
\end{equation}
\end{lemma}

\begin{proof}
Since $\partial_t\theta=\omega e^{\tau}$ and $\partial_t\rho=\rho e^{\tau}$, where $e^{\tau}=(T-t)^{-1}$, we have $\partial_tu=e^{(\alpha+1)\tau}(\mathcal D+\alpha)W$ with $\mathcal D:=\omega\partial_\theta+\rho\partial_\rho$, and then $\partial_t^2u=e^{(\alpha+2)\tau}(\mathcal D+\alpha+1)(\mathcal D+\alpha)W$. Also $\Delta u=e^{(\alpha+2)\tau}(W_{\rho\rho}+\frac{d-1}\rho W_\rho)$ and, since $p\alpha=\alpha+2$, $u^p=e^{(\alpha+2)\tau}W^p$. Expanding $(\mathcal D+\alpha+1)(\mathcal D+\alpha)=\mathcal D^2+(2\alpha+1)\mathcal D+\alpha(\alpha+1)$, with $\mathcal D^2=\omega^2\partial_\theta^2+2\omega\rho\partial_\theta\partial_\rho+\rho^2\partial_\rho^2+\rho\partial_\rho$, gives \eqref{eq:ssw}.
\end{proof}

\begin{definition}\label{def:profile}
A \emph{DSS profile of frequency $\omega$} is a real-analytic function $W$ on $\R\times(-\rho_1,\rho_1)$, for some $\rho_1>1$, which is $2\pi$-periodic in $\theta$, even in $\rho$, not identically zero, and solves \eqref{eq:ssw} for $0<\rho<\rho_1$.
\end{definition}

\begin{remark}\label{rem:noCSS}
The lateral boundary $\{|x|=T-t\}$ of the cone $\K_T$ is characteristic for the wave equation, and no characteristic of \eqref{eq:nlw} enters $\K_T$ through it. In the variables \eqref{eq:ssvar} this boundary is $\rho=1$, which is therefore characteristic for \eqref{eq:ssw}, and no boundary condition is imposed there. For the linear part of \eqref{eq:ssw} acting on $e^{in\theta}V(\rho)$, the Frobenius exponents at $\rho=1$ are $0$ and $s_c-\frac12-in\omega$ (Lemma~\ref{lem:hyp}). For $n\ne0$, the latter exponent gives a nonsmooth local branch, which must be absent in a smooth profile. For $n=0$, smoothness excludes the nonsmooth power or logarithmic terms, with logarithmic terms arising when $s_c-\frac12$ is an integer.

Note also that, for \eqref{eq:nlw} with critical exponent $s_c\le\frac32$, smooth $\tau$-independent profiles vanish identically, by the maximum principle (see also \cite[Proposition~2.2]{Ta}, which covers systems, with the strict inequality $s_c<\frac{3}{2}$). For the cases of \eqref{eq:cases} we were not able to find such profiles numerically either. This is one reason to look for time-periodic profiles. On the other hand, $\tau$-independent profiles which take the singular branch $(1-\rho)^{s_c-\frac12}$ at the light cone (with a logarithmic factor when $s_c-\frac12$ is an integer) exist for every energy-supercritical exponent, and they are positive and small: see Appendix~\ref{app:rough}, where they are used to construct rough solutions with finite-energy data whose supremum becomes infinite in finite time.
\end{remark}

\subsection{Blow-up from a profile} The following proposition, on the existence of DSS profiles, is the main result of this paper; Theorem~\ref{thm:main} follows from it by standard arguments.

\begin{proposition}\label{prop:profiles}
Let $(d,p)$ be one of the pairs in \eqref{eq:cases}. There exist $\omega^*>0$ satisfying \eqref{eq:omega} and a DSS profile $W^*$ of frequency $\omega^*$, real-analytic on $\R\times(-\rho_*,\rho_*)$ with $\rho_*=1.0011$ as in Theorem~\ref{thm:main}, such that $W^*(\theta+\pi,\rho)=-W^*(\theta,\rho)$ and
\begin{equation}\label{eq:profilebound}
\max_{0\le\rho\le1}|W^*(\theta,\rho)|\ge 1\qquad\text{for all }\theta\in\R.
\end{equation}
\end{proposition}

\begin{proof}[Proof of Theorem~\ref{thm:main}] This follows by well-known arguments, so we will be brief. The function $u^*$ defined in \eqref{eq:ustar} is real-analytic on the neighborhood $\{|x|<\rho_*(T-t)\}$ of $\K_T$, and solves \eqref{eq:nlw} there, since the computation in Lemma~\ref{lem:ss} is local. This proves (i).

(ii) Let $\varphi\in C_c^\infty([0,\rho_*))$ with $\varphi=1$ on $[0,1]$, and
\[
u_0(x):=\varphi(|x|/T)\,u^*(0,x),\qquad u_1(x):=\varphi(|x|/T)\,\partial_tu^*(0,x).
\]
These are radial, belong to $C_c^\infty(\R^{d})$, and coincide with $(u^*,\partial_tu^*)(0,\cdot)$ on $\{|x|\le T\}$.

Since the nonlinearity is a polynomial, the energy method in $H^s\times H^{s-1}$, $s>\frac d2+1$ (see e.g.\ \cite{So,SSbook}), gives a unique maximal smooth solution $u\in C^\infty([0,T^*)\times\R^{d})$. Moreover, by Gronwall's inequality, the solution can be continued as long as it stays bounded: if $T^*<\infty$ then $\limsup_{t\uparrow T^*}\|u(t)\|_{L^\infty}=\infty$. Standard energy estimates using finite speed of propagation show that $u=u^*$ in $\K_{T}\cap\{t<T^\ast\}$. In particular, $T^*\le T$ since $u^*$ is not bounded in $\K_T$. 

We now show that $T^*=T$, and that $u$ is bounded away from the origin. For $t<T^*$ the function $u(t)$ is smooth and, by finite speed of propagation, compactly supported, so the energy $E(u)$ is conserved; writing $E_0$ for its value, $\|\nabla u(t)\|_{L^2(\R^{d})}^2\le2E_0$ for all $t<T^*$. Since $u(t)$ is radial, for $x\neq0$ we have $u(t,x)=-\int_{|x|}^\infty\partial_ru(t,s\,x/|x|)\,ds$, thus
\begin{equation}\label{eq:radialbound}
|u(t,x)|\le\Big(\int_{|x|}^\infty s^{1-d}\,ds\Big)^{1/2}\Big(\int_{|x|}^\infty\big|\partial_ru(t,s\,x/|x|)\big|^2s^{d-1}\,ds\Big)^{1/2}\lesssim E_0^{1/2}|x|^{-\frac{d-2}2}
\end{equation}
for all $t<T^*$. Therefore $u$ is bounded on $[0,T^*)\times\{|x|\ge\eps\}$ for every $\eps>0$. In particular, if $T^\ast<T$ then $u$ is bounded on $[0,T^*)\times\R^{d}$ (since $u=u^*$ is bounded on $\K_T\cap\{t\le T^*\}$), which contradicts the continuation criterion discussed earlier. Hence $T^*=T$, and $u=u^*$ on $\K_T$. Moreover, a finite speed of propagation argument and the bounds \eqref{eq:radialbound} show that the solution extends smoothly to $([0,T]\times\R^{d})\setminus\{(T,0)\}$. Since $u=u^*$ on $\K_T$, the lower bound in (ii) follows from \eqref{eq:profilebound}, and the conclusions concerning the scale-invariant Lebesgue norms and the energy follow from \eqref{eq:ustar} by the change of variables $x=(T-t)y$. This proves (ii).

(iii) The function $\tilde u(t,x):=u(t,x')$ is a smooth solution of \eqref{eq:nlw} on $[0,T)\times\R^{d'}$, since $\Delta_x$ acts on it as $\Delta_{x'}$, and its data coincide with $(\chi u_0,\chi u_1)$ on $\{|x''|\le2T\}$. The existence of the maximal smooth solution $v$, with lifespan $T'$, and the continuation criterion hold in $\R^{d'}$ exactly as in $\R^{d}$. By finite speed of propagation, $v=\tilde u$ on $\{|x''|\le2T-t\}$ for $t<\min\{T,T'\}$, since the backward light cone with vertex at a point of this set meets $\{t=0\}$ inside $\{|x''|\le2T\}$. If $T'>T$, then $v$ would be continuous, hence bounded, on the compact set $[0,T]\times\{|x'|\le T\}\times\{|x''|\le T\}$, on which $v=\tilde u$ for $t<T$; but $\tilde u$ is unbounded there, since $u=u^*$ is unbounded on $\K_T$. Hence $T'\le T$, so that $v=\tilde u$ on $\{|x''|\le2T-t\}$ for all $t<T'$, and $\limsup_{t\uparrow T'}\|v(t)\|_{L^\infty(\R^{d'})}=\infty$ by the continuation criterion.
\end{proof}

The sign of the nonlinearity is used only to prove \eqref{eq:radialbound} through the coercivity of the energy. The rest of the paper is devoted to the proof of Proposition~\ref{prop:profiles}.

\section{From rotating to discretely self-similar profiles}\label{sec:origin}

This section explains how the approximate profiles were found. Since we were not able to find smooth real self-similar profiles, independent of $\tau$, we pursue instead the next simplest objects, namely the discretely self-similar (DSS) solutions, which are periodic in $\tau$. We describe a route from the rotating self-similar solutions of the complex equation, which are comparatively easy to find, to real DSS profiles, and show how the symmetries force the linearization to degenerate where the numerical branch reaches real profiles.

\subsection{Rotating profiles and symmetry-breaking}\label{sec:rot} For complex-valued solutions of $\partial_t^2u-\Delta u+|u|^{p-1}u=0$, the $U(1)$ symmetry allows \emph{rotating self-similar} solutions
\[
u(t,x)=(T-t)^{-\alpha-i\omega}\,Q\Big(\frac{|x|}{T-t}\Big),
\]
where $Q$, which is smooth across $\rho=1$, solves an ordinary differential equation. In self-similar variables, with $\theta=\omega\tau$ and $u=u_1+iu_2$, the profile of such a solution is
\[
(W_1,W_2)(\theta,\rho)=\big(\mathrm{Re}(e^{i\theta}Q(\rho)),\ \mathrm{Im}(e^{i\theta}Q(\rho))\big).
\]
It is $2\pi$-periodic and odd under $\theta\mapsto\theta+\pi$, and it satisfies $W_1=\partial_\theta W_2$ and $W_2=-\partial_\theta W_1$. Thus a rotating profile is a DSS profile of a two-component system, of a special kind: it contains only the Fourier modes $\pm1$ in $\theta$, and its time dependence is generated by the symmetry.

To connect rotating profiles with the real equation, consider the family of systems
\begin{equation}\label{eq:system}
\partial_t^2u_j-\Delta u_j+\partial_{u_j}F_\eps(u_1,u_2)=0,\quad j=1,2,\qquad F_\eps(u_1,u_2):=\frac{(u_1^2+\eps u_2^2)^{\frac{p+1}2}}{p+1},\qquad\eps>0 .
\end{equation}
For $\eps=1$ this is the complex equation. For every $\eps$, the subspace $\{u_2=0\}$ is invariant and carries the real equation, and so is $\{u_1=0\}$, where $u_2=\eps^{-\frac{p+1}{2(p-1)}}v$ with $v$ a solution of the real equation. The system is invariant under $u_1\mapsto-u_1$, and under the exchange $(u_1,u_2,\eps)\mapsto(u_2,u_1,\eps^{-1})$ combined with the scaling $u\mapsto\eps^{\frac{p+1}{2(p-1)}}u$. For $\eps\neq1$ the $U(1)$ symmetry is broken. We numerically continued the rotating profile as a DSS profile of \eqref{eq:system}, allowing all odd harmonics in $\theta$. Following the computed branch in both directions from $\eps=1$ suggests that it reaches the invariant subspaces at $\eps=1/p$ and $\eps=p$, through pitchfork bifurcations, with real DSS profiles as endpoints; see Figure~\ref{fig:bridge} (a), (b). These continuation and bifurcation claims are numerical observations.

\subsection{Symmetry-predicted transverse degeneracies at $\eps=1/p$ and $\eps=p$}
The symmetries force transverse degeneracies at these values. For the dimension $d$ and odd power $p$ under consideration, let $\mathcal P_{\omega,d,p}$ denote the linear part of the scalar profile equation \eqref{eq:ssw}, so that \eqref{eq:ssw} reads $\mathcal P_{\omega,d,p}W+W^p=0$. Let $V$ be a nonstationary real DSS profile satisfying $\mathcal P_{\omega,d,p}V+V^p=0$, and put $\mathcal L_V:=\mathcal P_{\omega,d,p}+pV^{p-1}$, with $\omega$ fixed. Translation invariance in $\theta$ gives
\[
\mathcal L_V\,\partial_\theta V=0 .
\]
Consider now the pure profiles $(W_1,W_2)=(0,\eps^{-\frac{p+1}{2(p-1)}}V)$ of \eqref{eq:system}. At these profiles the linearization in the direction of $W_1$ decouples from the rest, and since $\partial_{u_1}^2F_\eps(0,u_2)=\eps^{\frac{p-1}2}u_2^{p-1}$, it is $\mathcal P_{\omega,d,p}+\eps^{-1}V^{p-1}$. For $\eps=1/p$ it coincides with $\mathcal L_V$, and is therefore singular, with $\partial_\theta V$ in its kernel. This identifies $\eps=1/p$ as a candidate bifurcation point. Under suitable simple-kernel, transversality and smoothness hypotheses, a symmetry-compatible local branch can be parametrized in the form
\[
(W_1,W_2)=\Big(\delta\,\partial_\theta V+O(\delta^2),\ \eps^{-\frac{p+1}{2(p-1)}}V+O(\delta^2)\Big),\qquad\eps=\frac1p+O(\delta^2),
\]
where the expansion of $\eps$ is even in $\delta$ because of the symmetry $u_1\mapsto-u_1$. The leading transverse direction is proportional to the $\theta$-derivative of the pure profile, as suggested by the rotating-profile relations at $\eps=1$. On the other pure branch $(V,0)$, the transverse linearization is $\mathcal P_{\omega,d,p}+\eps V^{p-1}$, so the same degeneracy occurs at $\eps=p$, consistently with the exchange symmetry.

\begin{figure}[t]
\centering
\IfFileExists{fig_bridge.pdf}{\includegraphics[width=\textwidth]{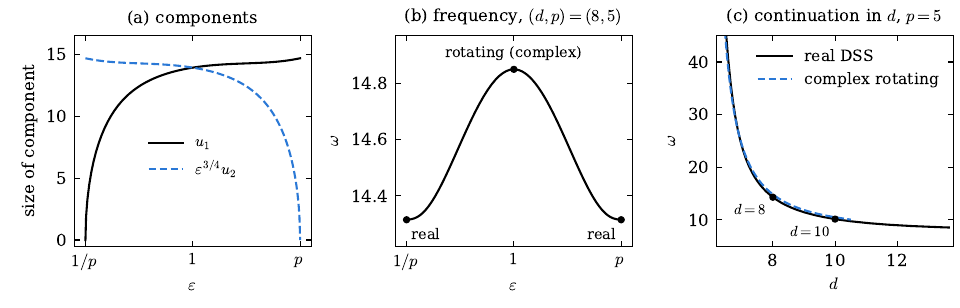}}{\fbox{\texttt{fig\_bridge.pdf}}}
\caption{Numerical continuation from rotating to real profiles. (a) Sizes (Euclidean norms of the coefficient arrays) of the two components along the branch of DSS profiles of \eqref{eq:system} for $(d,p)=(8,5)$, from the real profile at $\eps=1/p$, where $u_1=0$, through the rotating profile at $\eps=1$, to the real profile at $\eps=p$, where $u_2=0$. The second component is rescaled by $\eps^{\frac{p+1}{2(p-1)}}=\eps^{3/4}$. (b) The frequency along the same branch. (c) The frequencies of the real DSS family and of the rotating family under continuation in $d$, at $p=5$. The dots mark the real profiles at $d=8$ and $d=10$. Panels (a) and (b) were computed with $6$ odd harmonics and $40$ Chebyshev modes.}
\label{fig:bridge}
\end{figure}

\subsection{The profiles of Proposition~\ref{prop:profiles}}\label{sec:profiles}
We started from the rotating profile with $(d,p)=(8,5)$ and $\omega=14.8503$, computed from the ordinary differential equation for $Q$, and followed the branch of \eqref{eq:system} with $6$ odd harmonics and $40$ Chebyshev modes. The endpoint $\eps=p$, refined with $16$ harmonics and $64$ Chebyshev modes, is a numerical approximation to a real DSS profile for $(d,p)=(8,5)$ with $\omega=14.3108$. We checked it independently by time evolution in self-similar variables over one period. Continuation in $d$ at $p=5$ then gives profiles for other dimensions, for instance at $d=10$ and, at $d=7$, the profile $(7,5)$ of Proposition~\ref{prop:profiles}; see Figure~\ref{fig:bridge}(c). Along the way, the real and the rotating families stay close, and the computed frequencies grow rapidly as $d$ decreases below about $6.5$. Continuation in $p$ at fixed $d$, followed by continuation in $d$, gives the profile $(6,7)$ in the same way. 

The rotating profiles of the complex equation come in several families, which, for a given frequency $\omega$, differ by the slope at the light cone of the phase function $S$ of the WKB description $Q(\rho)\approx\omega^{\alpha}A(\rho)e^{i\omega S(\rho)}$, valid for large $\omega$. Each family can be continued to the real equation by the method of Section~\ref{sec:rot}. The profile $(9,3)$ of Proposition~\ref{prop:profiles}, with $\omega\approx40.26$, comes from the second family, found at $p=5$ by a Newton homotopy from the first family, continued in $p$ down to $p=3$ and in $d$ down to $d=9$, and brought to the real equation at $d=9$ by the continuation of Section~\ref{sec:rot}. The first family also reaches $p=3$ at $d=9$, but with $\omega\approx115$; at $d=10$ its frequency, $\omega\approx42.8$, is comparable to that of the second family. Figure~\ref{fig:profile} shows the three profiles. As the dimension decreases towards the critical dimension of each exponent, the frequencies grow and the profiles develop, in each Fourier mode $n$, a WKB oscillation $e^{\pm in\omega S(\rho)}$ which concentrates near the light cone: for $(9,3)$ the Chebyshev coefficients of the mode $n$ form a plateau up to the index $n_2\approx11n$ before decaying. This structure dictates the sizes of the approximate solutions in Table~\ref{tab:proof}, and it is also what limits the method for even smaller dimensions (Remark~\ref{rem:cost}).

The candidates obtained by continuation are accurate to a few digits only (the frequency of the $(9,3)$ candidate is $40.26024$, for instance), and have the (smaller) truncation sizes of the continuation rather than the sizes $N_W$ of Table~\ref{tab:proof} used in the proofs. They were extended to the sizes $N_W$ and refined by Newton's method on the Galerkin system $\pi^{\le N_W}F(U)=0$ (in the notation of Section~\ref{sec:modes}), in double-double arithmetic (about $32$ significant digits). These refined approximate solutions are the $\bar U$ of Section~\ref{sec:result}.

\section{Spectral formulation and the fixed-point problem}\label{sec:modes}

\subsection{Fourier decomposition and hypergeometric structure}
We write $W$ in terms of its complex-valued Fourier coefficients $V_n(\rho)$, $n\in2\Z+1$, as 
\begin{equation}\label{eq:fourier}
W(\theta,\rho)=\sum_{n\in2\Z+1}V_n(\rho)e^{in\theta},\qquad V_{-n}=\ov{V_n}.
\end{equation}
With this ansatz, \eqref{eq:ssw} is equivalent to the system
\begin{equation}\label{eq:modes}
L_nV_n+[W^p]_n=0,\qquad n\in2\N-1,
\end{equation}
where $2\N-1:=\{1,3,5,\dots\}$ is the set of odd natural numbers, $[f]_n:=\frac1{2\pi}\int_0^{2\pi}f(\theta)e^{-in\theta}\,d\theta$, and
\begin{equation}\label{eq:Ln}
L_n:=-n^2\omega^2+(2\alpha+1)in\omega+\alpha(\alpha+1)+\Big(2in\omega\rho-\frac{d-1}\rho+2(\alpha+1)\rho\Big)\partial_\rho-(1-\rho^2)\partial_\rho^2 .
\end{equation}
Since $W^p$ is real and odd under $\theta\mapsto\theta+\pi$, the equations for even $n$ are trivially satisfied, and those for negative $n$ are the complex conjugates of those for positive $n$. Throughout, we write
\begin{equation}\label{eq:sab}
s_n:=n\omega,\qquad a_n:=\frac{\alpha+is_n}2,\qquad b_n:=\frac{\alpha+1+is_n}2 .
\end{equation}

\begin{lemma}\label{lem:hyp}
Let $V(\rho)=\tilde V(\rho^2)$. Then $(L_nV)(\rho)=-4(H_n\tilde V)(\rho^2)$, where
\begin{equation}\label{eq:Hn}
H_n:=z(1-z)\partial_z^2+\Big(\frac d2-(a_n+b_n+1)z\Big)\partial_z-a_nb_n  \quad \text{ and } z:=\rho^2.
\end{equation}
\end{lemma}

\begin{proof}
With $z=\rho^2$ we have $\partial_\rho=2\rho\partial_z$ and $\partial_\rho^2=2\partial_z+4z\partial_z^2$. Hence $-(1-\rho^2)\partial_\rho^2=-4z(1-z)\partial_z^2-2(1-z)\partial_z$, and the first-order part of $L_n$ equals $(4in\omega z-2(d-1)+4(\alpha+1)z)\partial_z$. Summing up, the coefficient of $\partial_z$ is $-2(1-z)+4in\omega z-2(d-1)+4(\alpha+1)z=-2d+4(\alpha+\frac32+in\omega)z=-4\big(\frac d2-(a_n+b_n+1)z\big)$, since $a_n+b_n+1=\alpha+\frac32+in\omega$. Finally $4a_nb_n=(\alpha+in\omega)(\alpha+1+in\omega)=\alpha(\alpha+1)+(2\alpha+1)in\omega-n^2\omega^2$.
\end{proof}

Note that $H_n$ is a hypergeometric operator with exponents $\{0,1-\frac d2\}$ at $z=0$, $\{0,s_c-\frac12-is_n\}$ at $z=1$ and $\{a_n,b_n\}$ at $z=\infty$. The exponent $s_c-\frac12-is_n$ corresponds to the singular branch $(1-\rho)^{s_c-\frac12-in\omega}$ of Remark~\ref{rem:noCSS}.

\subsection{Chebyshev coefficients}

In the previous subsection we wrote an infinite system of ordinary differential equations for the Fourier coefficients $V_n(\rho)$, $n\in2\Z+1$. We now expand each $V_n$ in Chebyshev polynomials $T_k(2\rho^2-1)$, $k\ge0$, and write the system as an infinite system of algebraic equations for the Fourier--Chebyshev coefficients. The resulting matrix has an upper triangular structure, which we exploit to obtain uniform bounds on its inverse. First we introduce the function spaces.

For $\nu>1$ let $\ell^1_\nu$ denote the space of complex sequences $v=(v_k)_{k\ge0}$ with $\|v\|_\nu:=\sum_k\nu^k|v_k|<\infty$. To $v\in\ell^1_\nu$ we associate the function
\begin{equation}\label{eq:cheb}
V(\rho):=\sum_{k\ge0}v_kT_k(2\rho^2-1),
\end{equation}
where $T_k$ denotes the Chebyshev polynomial of degree $k$, $T_k(\cos\phi)=\cos k\phi$. The series converges absolutely and locally uniformly, together with all its derivatives, for $x=2\rho^2-1$ in the open Bernstein ellipse $\Ell_\nu:=\{\frac12(\zeta+\zeta^{-1}):1\le|\zeta|<\nu\}$ with foci $\pm1$; see \cite{Tr}. Since $\Ell_\nu\cap\R=(-\frac{\nu+\nu^{-1}}2,\frac{\nu+\nu^{-1}}2)$, $V$ is an even real-analytic function on $(-\rho_\nu,\rho_\nu)$ with
\begin{equation}\label{eq:rhonu}
\rho_\nu:=\tfrac12\big(\nu^{1/2}+\nu^{-1/2}\big)>1 .
\end{equation}
For $\nu=1.1$ one has $\rho_\nu>1.0011$. We write $\Lambda_n$ for the (infinite) matrix of $L_n$ acting on Chebyshev coefficients, so that $L_nV=\sum_j(\Lambda_nv)_jT_j(2\rho^2-1)$, and we set
\begin{equation}\label{eq:gam}
\begin{split}
&D_k:=-(k+a_n)(k+b_n),\\
&\gamma_n:=-(a_n+b_n)=-\alpha-\tfrac12-is_n,\qquad\gamma_n':=d-1-(a_n+b_n)=d-\alpha-\tfrac32-is_n.
\end{split}
\end{equation}

\begin{lemma}\label{lem:struct}
For every $k\ge0$,
\begin{equation}\label{eq:HT}
H_nT_k=D_kT_k+2k\Big(\gamma_n{\sum_{j<k,\ k-j\ \mathrm{even}}}'T_j+\gamma_n'{\sum_{j<k,\ k-j\ \mathrm{odd}}}'T_j\Big),
\end{equation}
where $T_k=T_k(x)$, $x=2z-1$, and the prime indicates that a term with $j=0$ carries the factor $\frac12$. Consequently $\Lambda_n=-4H_n$ is upper triangular, with
\[
(\Lambda_n)_{kk}=-4D_k,\qquad(\Lambda_n)_{jk}=\begin{cases}-8k\mu_j\gamma_n,& j<k,\ k-j\ \text{even},\\ -8k\mu_j\gamma_n',& j<k,\ k-j\ \text{odd},\end{cases}
\]
where $\mu_0:=\frac12$ and $\mu_j:=1$ for $j\ge1$. Moreover $|D_k|\ge k^2+\frac{s_n^2}4$.
\end{lemma}

\begin{proof}
In the variable $x$, $H_n=(1-x^2)\partial_x^2+(d-(a_n+b_n+1)(1+x))\partial_x-a_nb_n$. We use the Chebyshev equation $(1-x^2)T_k''=xT_k'-k^2T_k$ and the classical identities
\[
T_k'=2k{\sum_{j<k,\ k-j\ \mathrm{odd}}}'T_j,\qquad xT_k'=kT_k+2k{\sum_{j<k,\ k-j\ \mathrm{even}}}'T_j .
\]
They give $H_nT_k=-(k^2+(a_n+b_n)k+a_nb_n)T_k-(a_n+b_n)\,2k\sum'_{\rm even}T_j+(d-1-a_n-b_n)\,2k\sum'_{\rm odd}T_j$, which is \eqref{eq:HT}. The lower bound follows from $|k+a_n|^2=(k+\frac\alpha2)^2+\frac{s_n^2}4$ and $|k+b_n|^2=(k+\frac{\alpha+1}2)^2+\frac{s_n^2}4$.
\end{proof}

Thus, above the diagonal, $\Lambda_n$ coincides with the rank-two matrix
\[
\Big(-8k\mu_j\Big(\frac{\gamma_n+\gamma_n'}2+\frac{\gamma_n-\gamma_n'}2(-1)^{j}(-1)^{k}\Big)\Big)_{j,k}.
\]
By Lemma~\ref{lem:struct}, $\Lambda_n$ is an upper triangular matrix with nonzero diagonal entries. Consequently, each of its sections of finite order is invertible, and $\Lambda_n$ possesses an upper triangular inverse $\Lambda_n^{-1}$, whose columns have a finite number of nonzero entries. The boundedness of $\Lambda_n^{-1}$ on $\ell^1_\nu$ will be established in Section~\ref{sec:main}. The next lemma gives a recursive formula for the entries of $\Lambda_n^{-1} = -\frac{1}{4}H_n^{-1}$.

\begin{lemma}\label{lem:rec}
Fix $n$ and $k\ge0$, let $e_k$ denote the $k$-th canonical unit vector, and let $x:=H_n^{-1}e_k$ be the unique solution of $H_nx=e_k$ supported in $\{0,\dots,k\}$. Then $x_k=1/D_k$ and
\begin{equation}\label{eq:rec}
x_j=-\mu_j\frac{\gamma_n\zeta^{\rm e}_j+\gamma_n'\zeta^{\rm o}_j}{D_j},\quad\zeta_{k-1}=\Big(0,\frac{2k}{D_k}\Big)^{\!\top},\quad\zeta_{j-1}=M_j\zeta_j,\quad M_j:=\begin{pmatrix}0&1\\ 1-\frac{2j\gamma_n}{D_j}&-\frac{2j\gamma_n'}{D_j}\end{pmatrix}
\end{equation}
for $0\le j<k$, where $\zeta_j=(\zeta^{\rm e}_j,\zeta^{\rm o}_j)^\top$.
\end{lemma}

\begin{proof}
Let $\zeta^{\rm e}_j:=\sum_{i>j,\,i-j\ \mathrm{even}}2ix_i$ and $\zeta^{\rm o}_j:=\sum_{i>j,\,i-j\ \mathrm{odd}}2ix_i$. By Lemma~\ref{lem:struct}, row $j<k$ of $H_nx=e_k$ reads $D_jx_j+\mu_j(\gamma_n\zeta^{\rm e}_j+\gamma_n'\zeta^{\rm o}_j)=0$. Moreover $\zeta^{\rm e}_{j-1}=\zeta^{\rm o}_j$ and $\zeta^{\rm o}_{j-1}=\zeta^{\rm e}_j+2jx_j$.
\end{proof}

\begin{remark}[General dimension and nonlinearity]\label{rem:general}
The structure described above and the bounds of Section~\ref{sec:proof} hold for every $d\ge3$ and every odd $p\ge3$ satisfying \eqref{eq:supercritical}, and only involve $d$, $\alpha$ and $\gamma_n'-\gamma_n=d-1$. Only the computer-assisted part of the proof, in Section~\ref{sec:result}, depends on the particular values of $(d,p)$.
\end{remark}

\subsection{Setting}\label{sec:setting} Now that we have derived the algebraic system for the Fourier--Chebyshev coefficients, we introduce the Banach spaces in which we will solve it. The spaces are weighted $\ell^1$ spaces of sequences $V=(V_n)_{n\in\Z\times\N_0}$, $n=(n_1,n_2)$, where $\N_0:=\N\cup\{0\}$, with $n_1$ the Fourier index and $n_2$ the Chebyshev index. The weights are chosen so that the corresponding functions are real-analytic on a complex neighborhood of the interval $[-1,1]$ in $\rho$, and so that the convolution corresponding to multiplication of functions is bounded.

Fix $\nu>1$ and $\nu_\theta>1$. Let $\ell^1_{\nu_\theta,\nu}$ be the space of complex sequences $V=(V_n)_{n\in\Z\times\N_0}$, $n=(n_1,n_2)$, with
\[
\|V\|_{\ell^1}:=\sum_{n \in \Z\times\N_0}\nu_\theta^{|n_1|}\nu^{n_2}|V_n|<\infty.
\]
Moreover, let $\ell^1$ be the closed subspace of $\ell^1_{\nu_\theta,\nu}$ defined as 
\begin{equation*}
  \ell^1 := \{ V \in \ell^1_{\nu_\theta,\nu} : V_n = 0 \text{ for } n_1 \text{ even} \}.
\end{equation*}
 A sequence $W$ represents the function $W(\theta,\rho):=\sum_nW_ne^{in_1\theta}T_{n_2}(2\rho^2-1)$. In the rest of the paper we identify the function and its coefficient representation in $\ell^1.$ We say that $W$ is real if $W_{-n_1,n_2}=\ov{W_{n_1,n_2}}$ for all $n$, that is, if this function is real. Products of functions correspond to discrete convolutions of their coefficients, defined as
\[
(V*W)_n:=\frac12\sum_{m_1+p_1=n_1}\Big(\sum_{m_2+p_2=n_2}V_mW_p+\sum_{|m_2-p_2|=n_2}V_mW_p\Big),
\]
and we write $W^k$ for the $k$-th power of $W$. Moreover, $\ell^1_{\nu_\theta,\nu}$ is a Banach algebra under this convolution, that is,
\begin{equation}
  \|V*W\|_{\ell^1}\le\|V\|_{\ell^1}\|W\|_{\ell^1}, \quad \text{ for all } V,W\in\ell^1_{\nu_\theta,\nu}.
\end{equation}
Note that $\ell^1$ itself is not stable under the convolution. However, since $p$ is odd, if $W \in \ell^1$, then $W^p \in \ell^1$, and $V*W \in \ell^1$ for all $V \in \ell^1_{\nu_\theta,\nu}$ such that $V_n = 0$ for $n_1$ odd.

For odd $n_1\in\Z$ and $\omega\in\C$, we emphasize the dependence of $\Lambda_{n_1}$ on $\omega$ and write $\Lambda_{n_1}(\omega)$ for the upper triangular matrix of Lemma~\ref{lem:struct}, and $$(\Lambda(\omega)W)_{n_1,\cdot}:=\Lambda_{n_1}(\omega)W_{n_1,\cdot}.$$

Let us fix a truncation size $N=(N_1,N_2)\in\N^2$ with $N_1$ odd. This represents the size of the Galerkin projection which will be used for rigorous numerical computations.

Let us define $\ell:\ell^1\to\C$ as
\begin{align}
  \ell(W) := \frac1{2i}\sum_{n_2 = 0}^{N_2} \big( W_{1,n_2}-W_{-1,n_2}\big),
\end{align}
so that $\ell(W)=\sum_{n_2=0}^{N_2}\mathrm{Im}(W_{1,n_2})$ if $W$ is real.
Let $X := \C \times \ell^1$ with the norm $\|U\|_X := |\omega| + \|W\|_{\ell^1}$ for $U = (\omega, W)$, and
\begin{align}\label{eq:F}
  F(U)  := \begin{pmatrix}
    \ell(W)\\
    \Lambda(\omega)W+W^p
  \end{pmatrix}.
\end{align}
Thus $F(U)=0$ if and only if $\ell(W)=0$ and $W$ solves \eqref{eq:modes}, that is, the profile equation \eqref{eq:ssw} with frequency $\omega$.

We apply the Newton--Kantorovich theorem to the correction variable $\Psi$ defined by $U=\bar U+A\Psi$, where $A$ approximates the inverse of $DF(\bar U)$. In this variable, we seek a fixed point of $\Psi\mapsto\Psi-F(\bar U+A\Psi)$, whose derivative is $I-DF(\bar U+A\Psi)A$. This explains why $A$ acts on the right in the bounds below. The following formulation follows from the usual Newton--Kantorovich theorem applied to $\Psi\mapsto F(\bar U+A\Psi)$ at $\Psi=0$; see, e.g., \cite{ArKoTe05,DLM,NakPluWat19,Arioli_Koch_2023}. Here $\mathcal B(X_2)$ denotes the space of bounded linear operators on $X_2$, with the operator norm, and $B_r(0)$ denotes the closed ball of radius $r$ centered at $0$ in $X_2$.

\begin{proposition}[Newton--Kantorovich]\label{th : radii polynomial}
	Let $(X_1,\|\cdot\|_{X_1})$ and $(X_2,\|\cdot\|_{X_2})$ be Banach spaces, and let $F : X_1 \to X_2$ be Fréchet differentiable. Moreover, let $\bar{U} \in X_1$, and let ${A} : X_2 \to X_1$ be an injective bounded linear operator. 
	
	Assume there exist nonnegative constants $Y$ and $Z_1$, and a nonnegative non-decreasing function $Z_2 : (0, \infty) \to [0,\infty)$, such that:
	\begin{equation}\label{eq : bounds general radii polynomial}
	\begin{aligned}
	\|{F}(\bar{U})\|_{X_2} & \le {Y}, \\
	\|I - D{F}(\bar{U}) A\|_{\mathcal{B}(X_2)} &\le {Z}_1,\\
	\|\left({D}{F}(\bar{U} + A h) - D{F}(\bar{U})\right)A\|_{\mathcal{B}(X_2)} &\le {Z}_2(r)r, \qquad \text{for all } h \in {B_r(0)} \subset X_2 \text{ and all } r>0.
	\end{aligned}  
	\end{equation}
	If there exists a radius $r>0$ such that the radii polynomial conditions are satisfied:
	\begin{equation}\label{condition radii polynomial}
	\frac{1}{2}{Z}_2(r)r^2 - (1-{Z}_1)r + {Y} < 0 \qquad \text{and} \qquad {Z}_1 + {Z}_2(r)r < 1,
	\end{equation}
	then there exists a unique $\tilde{\Psi} \in {B_r}(0) \subset X_2$ such that ${F}(\bar{U} + A \tilde{\Psi})=0$. 
\end{proposition}

\begin{remark}\label{rem : the solution is real}
  For $U=(\omega,W)\in X$ let $\mathcal CU:=(\omega',W')$, where $\omega'$ is the complex conjugate of $\omega$ and $W'_n:=\ov{W_{-n_1,n_2}}$. We say that $U$ is real if $\mathcal CU=U$, that is, if $\omega\in\R$ and $W$ is real. The map $\mathcal C$ is an isometry of $X$, and
\begin{equation}\label{eq:sym}
F(\mathcal CU)=\mathcal CF(U),
\end{equation}
since $W'$ represents the complex conjugate of the function $W$ and $\Lambda_{-n_1}(\omega')=\ov{\Lambda_{n_1}(\omega)}$. Consequently, if $U$ is real, then $F(U)$ is real. Moreover, by \eqref{eq:sym}, if $U$ is a zero of $F$, then so is $\mathcal CU$. Now, assume that $\bar U$ is real and that $\mathcal{C}A = A\mathcal{C}$. If $\tilde\Psi$ is given by Proposition~\ref{th : radii polynomial}, then $\mathcal{C}\tilde\Psi \in B_r(0)$ and $F(\bar U + A\mathcal{C}\tilde\Psi) = \mathcal{C}F(\bar U + A\tilde\Psi) = 0$. Using the uniqueness statement of Proposition~\ref{th : radii polynomial}, we get $\mathcal{C}\tilde\Psi = \tilde\Psi$, that is, the zero $\tilde U=\bar U+A\tilde\Psi$ of $F$ is real.
\end{remark}

Our goal is to apply Proposition~\ref{th : radii polynomial} to the map $F$ defined in \eqref{eq:F}. The approximate solution $\bar U$ is a numerical approximation of a zero of $F$, and the operator $A$ is an approximate inverse of $DF(\bar U)$. The bounds $Y$, $Z_1$ and $Z_2(r)$ will be computed rigorously, by combining analytical estimates with interval arithmetic for the finite-dimensional computations. The radii polynomial conditions \eqref{condition radii polynomial} will be verified rigorously by interval arithmetic. Note that $\Lambda(\omega)$ is unbounded on $\ell^1$, so that $F$ does not map $X$ into itself: we choose the Banach spaces $X_1 := \mathbb C \times \Lambda(\bar \omega)^{-1} \ell^1$ and $X_2 := X$ to apply Proposition~\ref{th : radii polynomial}, where $X_1$ is normed by $\|(\omega,W)\|_{X_1}:=|\omega|+\|\Lambda(\bar\omega)W\|_{\ell^1}$. Since $\Lambda(\bar\omega)^{-1}$, $\partial_\omega\Lambda(\bar\omega)\Lambda(\bar\omega)^{-1}$ and $\mathcal N^2\Lambda(\bar\omega)^{-1}$, where $(\mathcal NW)_n=n_1W_n$, are bounded on $\ell^1$ (Lemmas~\ref{lem:kappa} and~\ref{lem:kappa prime} below), $X_1$ is a Banach space continuously embedded in $X$, $A$ maps $X$ boundedly into $X_1$, and $F:X_1\to X$ is a polynomial map, hence Fr\'echet differentiable, by \eqref{eq:dLam}. In particular, the bounds of Proposition~\ref{th : radii polynomial} are computed in the Banach space $X = \mathbb C \times \ell^1.$

\subsection{Construction of the approximate objects}

Recall that $N = (N_1, N_2) \in \mathbb{N}^2$, with $N_1$ odd, is the size of the Galerkin truncation. For an index $m=(m_1,m_2)$ we write $m\le N$ if $|m_1|\le N_1$ and $m_2\le N_2$. Let
\begin{align}
  (\pi^{\leq N}W)_n := \begin{cases}
    W_n, & n\le N, \\
    0, & \text{otherwise},
  \end{cases}
\end{align}
and $\pi^{> N} := I - \pi^{\leq N}$. On $X$, we slightly abuse notation and write $\pi^{\le N}(\omega,W):=(\omega,\pi^{\le N}W)$ and $\pi^{>N}(\omega,W):=(0,\pi^{>N}W)$. That is,  $g\in\ell^1$ is identified with $(0,g)\in X$. Moreover $\pi_m$ denotes the projection onto the coordinate $m$, with $e_m$ the corresponding unit sequence in $\ell^1$, and $\pi_\omega U:=\omega$, $\pi_WU:=W$ for $U=(\omega,W)$.

Let $\bar U = (\bar \omega, \bar W)\in X$ be real, with $\bar W=\pi^{\leq N_W}\bar W$ for some $N_W\in\N^2$: it is a numerical approximation of a zero of $F$. We do not assume that $N_W\le N$, so that the finite part $A_N$ of the approximate inverse below may involve fewer Fourier modes than $\bar W$.
Let $A : X \to X$ be defined as
\begin{align}\label{eq:A}
  A := A_N + \Lambda(\bar \omega)^{-1}\pi^{>N},
\end{align}
where $A_N = \pi^{\leq N} A_N\pi^{\leq N}$, that is, $A_N$ is a finite-rank operator, and $\Lambda(\bar \omega)^{-1}\pi^{>N}$ acts on the component $W$. In practice, $A_N$ is a numerical approximation of the inverse of the Galerkin projection $\pi^{\leq N} DF(\bar{U}) \pi^{\leq N}$. Following Proposition~\ref{th : radii polynomial}, we look for a zero of $F$ of the form
\[
U=\bar U+A\Psi,\qquad\Psi\in X,
\]
that is, for a fixed point of $T(\Psi):=\Psi-F(\bar U+A\Psi)$.

The approximate objects $\bar U$ and $A$ being fixed, we now concentrate on the computation of the bounds $Y$, $Z_1$ and $Z_2(r)$ of Proposition~\ref{th : radii polynomial}. This is the purpose of the next section.

\section{Computation of the bounds}\label{sec:proof}

We now verify the three bounds required by Proposition~\ref{th : radii polynomial}. The bound $Y$ measures the residual of the approximate solution, $Z_1$ measures the defect of the approximate inverse, and $Z_2(r)r$ controls the variation of the derivative of $\Psi\mapsto F(\bar U+A\Psi)$ on the ball of radius $r$. The residual and the contributions involving the finite operator $A_N$ are enclosed by interval arithmetic. To control the remaining columns, we derive estimates that are uniform in the Fourier and Chebyshev indices. Section~\ref{sec:result} combines these bounds to verify the radii polynomial conditions.

\subsection{Some preliminary computations}

The entries of $\Lambda_{n_1}(\omega)$ grow polynomially; hence, for every $W\in\ell^1$, $\Lambda(\omega)W$ is a well-defined sequence, which need not belong to $\ell^1$. The entries are polynomials of degree at most two in $\omega$: with $(\mathcal NW)_n:=n_1W_n$, for $\varsigma\in\C$ and $j<k$,
\begin{equation}\label{eq:dLam}
\Lambda(\bar\omega+\varsigma)=\Lambda(\bar\omega)+\varsigma\,\partial_\omega\Lambda(\bar\omega)-\varsigma^2\mathcal N^2,\qquad
\begin{aligned}
(\partial_\omega\Lambda_{n_1})_{kk}&=2in_1(2k+\alpha+\tfrac12+in_1\omega),\\
(\partial_\omega\Lambda_{n_1})_{jk}&=8in_1k\mu_j .
\end{aligned}
\end{equation}
From now on $\bar\omega>0$ is fixed, and we write $\Lambda:=\Lambda(\bar\omega)$, $\partial_\omega\Lambda:=\partial_\omega\Lambda(\bar\omega)$. The operator $\Lambda(\bar\omega)^{-1}$ has the blocks $\Lambda_{n_1}(\bar\omega)^{-1}$, the upper triangular matrices with columns $-\frac14H_{n_1}^{-1}e_k$ of Lemma~\ref{lem:rec}; we will show that it is bounded on $\ell^1$ (Lemma~\ref{lem:kappa}).

\medskip\noindent We will use the following facts.
\begin{enumerate}[(a)]
\item \emph{Derivative.} For $U=(\omega,W)$ and a direction $(\varsigma,h)$,
\begin{equation}\label{eq:DF}
DF(U)(\varsigma,h)=\begin{pmatrix}\ell(h)\\ \Lambda(\omega)h+pW^{p-1}*h+\varsigma\,\partial_\omega\Lambda(\omega)W\end{pmatrix},\qquad\partial_\omega\Lambda(\bar\omega+\varsigma)=\partial_\omega\Lambda-2\varsigma\,\mathcal N^2 .
\end{equation}
\item \emph{Supports.} $\Lambda$, $\partial_\omega\Lambda$ and $\Lambda^{-1}$ are diagonal in $n_1$ and upper triangular in $n_2$. Hence, for each of these operators $B$ and every $m\in\N^2$,
\begin{equation}\label{eq:tri}
B\,\pi^{\le m}=\pi^{\le m}B\,\pi^{\le m}.
\end{equation}
In particular $\Lambda\bar W$, $\partial_\omega\Lambda\bar W$ and $\mathcal N^2\bar W$ belong to $\pi^{\le N_W}\ell^1$.
\item \emph{Norms.} For a matrix $B=(B_{n,m})$ acting on $\ell^1_{\nu_\theta,\nu}$,
\begin{equation}\label{eq:opnorm}
\|B\|_{\ell^1}=\sup_m\|B\pi_m\|_{\ell^1},\qquad\|B\pi_m\|_{\ell^1}=\nu_\theta^{-|m_1|}\nu^{-m_2}\sum_n\nu_\theta^{|n_1|}\nu^{n_2}|B_{n,m}| .
\end{equation}
Similarly, the norm of an operator on $X$, or from $X$ to $\ell^1$ or $\C$, is the supremum of the norms of its columns. We write $\|\cdot\|_{\mathcal B(X)}$ for all these operator norms on $X$, and $\|\cdot\|_{\ell^1}$ for the norm of an operator on $\ell^1$.
\end{enumerate}

\subsection{Computation of the $Y$ bound}

   By \eqref{eq:F} and the definition of the norm of $X$,
   \begin{equation}\label{eq:Ybound}
   \|F(\bar U)\|_X=|\ell(\bar W)|+\|\Lambda(\bar\omega)\bar W+\bar W^p\|_{\ell^1},
   \end{equation}
   and $Y$ is an upper bound of the right-hand side. Since $\Lambda\bar W\in\pi^{\le N_W}\ell^1$ by (b), and $\bar W^p=\pi^{\le pN_W}\bar W^p$, the right-hand side of \eqref{eq:Ybound} is the norm of a finite sequence: its coefficients are computed with interval arithmetic, by the fast Fourier transform for the product $\bar W^p$. In particular, the part $\|\bar W^p-\pi^{\le N_W}\bar W^p\|_{\ell^1}$ of the residual, which is not in the Galerkin projection used to compute $\bar W$, is taken into account exactly.

\subsection{Computation of the $Z_1$ bound}

We estimate $I-DF(\bar U)A$ by separating its input columns inside and outside the Galerkin truncation. For the finite input columns, $Z_{1,0}$ bounds the defect within the truncation, while $Z_{1,1}$ accounts for the output outside it and for the truncation of $p\bar W^{p-1}$. The bound $Z_{1,2}$ controls all remaining input columns. In the rest of the paper, for $V\in\ell^1_{\nu_\theta,\nu}$, we write $M_V$ for multiplication by $V$, that is, $M_VW:=V*W$. The construction of $A$ in \eqref{eq:A} then gives the following bounds.

\begin{lemma}\label{lem:Z1}
Given $N_0 \in \mathbb N^2$, let $\bar{V} := \pi^{\leq N + N_0}(p\bar W^{p-1})$ be fixed, and let $DF_{\bar V}(\bar U)$ denote the operator obtained from $DF(\bar U)$ in \eqref{eq:DF} by replacing $p\bar W^{p-1}$ with $\bar V$. Let $Z_{1,0}$, $Z_{1,1}$ and $Z_{1,2}$ be bounds satisfying
\begin{align*}
  Z_{1,0} &\geq \|\pi^{\leq N} - \pi^{\leq N} D{F}_{\bar V}(\bar{U}) A_N \pi^{\leq N} \|_{\mathcal{B}(X)},\\
  Z_{1,1} &\geq \|\pi^{> N}  M_{\bar V} \pi_W A_N \pi^{\leq N} \|_{\mathcal{B}(X)} + \|\pi^{>N}\partial_\omega\Lambda\bar W\|_{\ell^1}\|\pi_\omega A_N\pi^{\leq N}\|_{\mathcal{B}(X)}\\
  &\quad+ \|\bar{V} - p\bar{W}^{p-1}\|_{\ell^1} \| \pi_W A_N \pi^{\leq N} \|_{\mathcal{B}(X)},\\
  Z_{1,2} & \geq \sup_{n\not\le N}\frac{|\ell(\Lambda^{-1}e_n)|+\|M_{p\bar W^{p-1}}\Lambda^{-1}e_n\|_{\ell^1}}{\|e_n\|_{\ell^1}}.
\end{align*}
Then, $Z_1 \geq \|I - D{F}(\bar{U}) A\|_{\mathcal{B}(X)}$ with $Z_1 := \max\{Z_{1,0} + Z_{1,1}, ~  Z_{1,2}\}$. Moreover, the last condition is satisfied by
\begin{equation}\label{eq:Z12product}
Z_{1,2}=\Big(\|\bar V\|_{\ell^1} + \|p\bar{W}^{p-1} - \bar{V}\|_{\ell^1} + \frac{1}{2\nu_\theta}\Big) \|\Lambda^{-1} \pi^{> N} \|_{\ell^1}.
\end{equation}

\end{lemma}

\begin{proof}
  By definition of the $Z_1$ bound and $A$ in \eqref{eq:A}, we need to estimate the following quantity:
  \begin{align*}
    \|I - D{F}(\bar{U}) A\|_{\mathcal{B}(X)} &= \max\{ \|\pi^{\leq N} - D{F}(\bar{U}) A \pi^{\leq N} \|_{\mathcal{B}(X)}, ~ \|\pi^{> N} - D{F}(\bar{U}) A \pi^{> N} \|_{\mathcal{B}(X)}\}\\
    & = \max\{ \|\pi^{\leq N} - D{F}(\bar{U}) A_N \pi^{\leq N} \|_{\mathcal{B}(X)}, ~ \|\pi^{> N} - D{F}(\bar{U}) \Lambda^{-1} \pi^{> N} \|_{\mathcal{B}(X)}\}.
  \end{align*}
 Focusing on the first term of the right-hand side, we write $DF(\bar U)=DF_{\bar V}(\bar U)+M_{p\bar W^{p-1}-\bar V}\pi_W$ and use \eqref{eq:DF} and \eqref{eq:tri}, which give
\[
\pi^{>N}DF_{\bar V}(\bar U)A_N\pi^{\leq N}=\pi^{>N}M_{\bar V}\pi_WA_N\pi^{\leq N}+\pi^{>N}\partial_\omega\Lambda\bar W\,\pi_\omega A_N\pi^{\leq N},
\]
where the last operator maps $\Phi$ to $(\pi_\omega A_N\Phi)\,\pi^{>N}\partial_\omega\Lambda\bar W$ (it vanishes if $N_W\le N$, by (b)). Hence
  \begin{align*}
   \pi^{\leq N} - D{F}(\bar{U}) A_N \pi^{\leq N} &= \big(\pi^{\leq N} - \pi^{\leq N} D{F}_{\bar V}(\bar{U}) A_N \pi^{\leq N}\big) - \pi^{> N}  M_{\bar V} \pi_W A_N \pi^{\leq N}\\
   &\quad -\pi^{>N}\partial_\omega\Lambda\bar W\,\pi_\omega A_N\pi^{\leq N} - M_{p\bar W^{p-1}-\bar V}\pi_W A_N \pi^{\leq N},
  \end{align*}
 and, since the norms of the last two operators are at most $\|\pi^{>N}\partial_\omega\Lambda\bar W\|_{\ell^1}\|\pi_\omega A_N\pi^{\leq N}\|_{\mathcal{B}(X)}$ and $\|\bar{V} - p\bar{W}^{p-1}\|_{\ell^1} \| \pi_W A_N \pi^{\leq N} \|_{\mathcal{B}(X)}$,
\begin{align*}
  \|\pi^{\leq N} - D{F}(\bar{U}) A_N \pi^{\leq N} \|_{\mathcal{B}(X)} \leq Z_{1,0} + Z_{1,1}.
\end{align*}

Now, focusing on the second term $\|\pi^{> N} - D{F}(\bar{U}) \Lambda^{-1} \pi^{> N} \|_{\mathcal{B}(X)}$, let $g \in \pi^{>N}\ell^1$. Using \eqref{eq:DF}, we have
\begin{align*}
  \left(\pi^{> N} - D{F}(\bar{U}) \Lambda^{-1} \pi^{> N}\right) \begin{pmatrix} 0 \\ g \end{pmatrix} = - \begin{pmatrix} \ell(\Lambda^{-1}g) \\ M_{p\bar{W}^{p-1}}\Lambda^{-1}g \end{pmatrix}.
\end{align*}
Taking $g=e_n$, $n\not\le N$, and using \eqref{eq:opnorm}, we obtain
\[
  \|\pi^{> N} - D{F}(\bar{U}) \Lambda^{-1} \pi^{> N} \|_{\mathcal{B}(X)} = \sup_{n\not\le N}\frac{|\ell(\Lambda^{-1}e_n)|+\|M_{p\bar W^{p-1}}\Lambda^{-1}e_n\|_{\ell^1}}{\|e_n\|_{\ell^1}}\leq Z_{1,2}.
\]
Finally, $|\ell(W)| \leq \frac{1}{2\nu_\theta}\|W\|_{\ell^1}$ and $\|M_{p\bar W^{p-1}}W\|_{\ell^1}\le\|p\bar W^{p-1}\|_{\ell^1}\|W\|_{\ell^1}\le\big(\|\bar V\|_{\ell^1}+\|p\bar{W}^{p-1} - \bar{V}\|_{\ell^1}\big)\|W\|_{\ell^1}$ for all $W \in \ell^1$. Hence each quotient above is at most $\big(\|\bar V\|_{\ell^1} + \|p\bar{W}^{p-1} - \bar{V}\|_{\ell^1} + \frac{1}{2\nu_\theta}\big)\|\Lambda^{-1}\pi_n\|_{\ell^1}$, and \eqref{eq:Z12product} follows from \eqref{eq:opnorm}.
\end{proof}

By the above lemma, $Z_{1,0}$ and $Z_{1,1}$ can be computed with interval arithmetic, since they are matrix norms or norms of finite sequences. The quantities $\|\bar V\|_{\ell^1}$ and $\|p\bar W^{p-1} - \bar V\|_{\ell^1}$ are enclosed with interval arithmetic, as for \eqref{eq:Ybound}, since $p\bar W^{p-1}=\pi^{\le(p-1)N_W}(p\bar W^{p-1})$ is a finite sequence. For the computation of $Z_{1,2}$, we use \eqref{eq:Z12product} for all but finitely many columns, which requires the norm of $\Lambda^{-1} \pi^{> N}$. This norm is computed in Section~\ref{sec:main}; in Section~\ref{sec:exact}, we also compute directly the quotients of Lemma~\ref{lem:Z1} for the columns on which \eqref{eq:Z12product} is far from sharp.

\subsubsection{Computation of the norm of $\Lambda^{-1}\pi^{>N}$}\label{sec:main}

To bound $\|\Lambda^{-1}\pi^{>N}\|_{\ell^1}$, choose $K=(K_1,K_2)\in\N^2$ with $K_1$ odd, $N_1<K_1$ and $N_2<K_2$. We divide the columns into three regions: the finite set $n\le K$, the columns with $|n_1|\le K_1$ and $n_2>K_2$, and the columns with $|n_1|>K_1$. We first enclose the finite set by verified matrix inversions. For the two infinite regions, the recursion of Lemma~\ref{lem:rec} gives uniform estimates: at large Fourier indices it controls entire columns, while at large Chebyshev indices it reduces the low rows to two previously enclosed columns.

For the finite set, we construct approximate inverses for all the matrices $\Lambda_{n_1}$ with $|n_1| \leq K_1$. Namely, we construct an upper triangular matrix $X_{n_1}$ as an approximate inverse of the Galerkin truncation of order $K_2$ of $\Lambda_{n_1}$, denoted $\Lambda_{n_1}^{(K_2)}$, and let
\[
\eps_{n_1}:=\|I_{K_2}-\Lambda_{n_1}^{(K_2)}X_{n_1}\|_{\ell^1_\nu},
\]
where $I_{K_2}$ is the identity on $\C^{K_2+1}.$
By \eqref{eq:tri}, the computation of $\eps_{n_1}$ is finite-dimensional, of size $K_2+1$. The matrices $X_{n_1}$ allow us to control the columns $n$ of $\Lambda^{-1}$ for all $n \leq K$. Indeed, since $\Lambda_{n_1}$ is upper triangular, the inverse of its section of order $K_2$ is the section of order $K_2$ of $\Lambda_{n_1}^{-1}$. Then, assuming that $\eps_{n_1}<1$ and  using a Neumann series argument, we readily have
\begin{equation}\label{eq: finite number of columns}
  \|\Lambda^{-1} \pi_n\|_{\ell^1} \leq  \|X_{n_1} \pi_{n_2}\|_{\ell^1_\nu} + \frac{\eps_{n_1}}{1- \eps_{n_1}}\|X_{n_1}\|_{\ell^1_\nu} =: \lambda_n \quad  \text{ for all } n \leq K.
\end{equation}
Similarly, using \eqref{eq:tri} once more, we have
\begin{equation}\label{eq: finite number of columns derivative}
  \|\partial_\omega\Lambda\,\Lambda^{-1} \pi_n\|_{\ell^1} \leq  \|\partial_\omega\Lambda_{n_1}X_{n_1} \pi_{n_2}\|_{\ell^1_\nu} + \frac{\eps_{n_1}}{1- \eps_{n_1}}\|\partial_\omega\Lambda_{n_1}X_{n_1}\|_{\ell^1_\nu} =: \lambda'_n \quad  \text{ for all } n \leq K,
\end{equation}
which will be used in Section~\ref{sec:Z2} for the computation of the $Z_2$ bound. The above is computed entirely with interval arithmetic. This settles the columns $n \leq K$.

For the remaining columns, we rely on the recursion derived in Lemma~\ref{lem:rec}. In the rest of this section, $n_1$ is odd, and $s_{n_1}$, $a_{n_1}$, $D_j$, $\gamma_{n_1}'$, $M_j$ and $\zeta_j$ are the quantities defined in \eqref{eq:sab}, \eqref{eq:gam} and Lemma~\ref{lem:rec} for $n = n_1$ and $\omega = \bar\omega$. In particular, we recall that
\begin{equation}\label{eq:modD}
|D_j|^2=\Big(\big(j+\tfrac\alpha2\big)^2+\tfrac{s_{n_1}^2}4\Big)\Big(\big(j+\tfrac{\alpha+1}2\big)^2+\tfrac{s_{n_1}^2}4\Big) \quad \text{ and } \quad |\gamma_{n_1}'|=\big(g^2+s_{n_1}^2\big)^{\frac{1}{2}},
\end{equation}
where $g:=d-\alpha-\frac32$ denotes the real part of $\gamma_{n_1}'$.
Moreover, we endow $\C^2$ with the following norm
\[
|\zeta|_\nu:=\max\Big\{|\zeta^{\rm e}|,\ \frac2{\nu-1}\,|\zeta^{\rm e}+\zeta^{\rm o}|\Big\}, \quad \text{ for all } \zeta=(\zeta^{\rm e},\zeta^{\rm o})^\top\in\C^2,
\]
and we define 
\begin{equation}\label{eq:beta}
\beta_{n_1}:=\frac{4\big(2(d-1)+(\nu-1)|\gamma_{n_1}'|\big)}{(\nu-1)^2}.
\end{equation}
The norm $|\cdot|_\nu$ is chosen such that the matrices $M_j$ of Lemma~\ref{lem:rec} have norm at most $\frac{1+\nu}2 < \nu$ as soon as $j$ or $|n_1|$ is large enough. This is the purpose of the next lemma, which requires $K$ to be chosen such that
\begin{align}
\frac{4(d-1)}{(\nu-1)(K_2-1)}+\Big(1+\frac{4g}{K_2-1}+\frac{4(d-1)g}{(K_2-1)^2}+\frac{\alpha(\alpha+1)g}{(K_2-1)^3}\Big)^{\frac{1}{2}}&\le\frac{1+\nu}2,\label{eq:H1}\\
\frac{4(d-1)}{(\nu-1)s_{K_1+2}}+\Big(1+\frac{2.6\,g}{s_{K_1+2}}+\frac{4(d-1)g}{s_{K_1+2}^2}+\frac{2.6\,\alpha(\alpha+1)g}{s_{K_1+2}^3}\Big)^{\frac{1}{2}}&\le\frac{1+\nu}2.\label{eq:H2}
\end{align}
For $\nu=1.1$ and $\omega=\bar\omega$, \eqref{eq:H1} and \eqref{eq:H2} are satisfied as soon as $K_2\ge6662$ and $K_1\ge163$ for $(d,p)=(9,3)$, where $g=\frac{13}2$, $K_2\ge5002$ and $K_1\ge203$ for $(7,5)$, where $g=5$, and $K_2\ge4168$ and $K_1\ge205$ for $(6,7)$, where $g=\frac{25}6$. Since the first terms of \eqref{eq:H1} and \eqref{eq:H2} must be smaller than $\frac{\nu-1}2$, one needs $K_2\gtrsim8(d-1)/(\nu-1)^2$ and $K_1\gtrsim K_2/\omega$: the weight $\nu$ cannot be taken close to $1$. In the proofs of Section~\ref{sec:result}, $K_2$ is larger (Table~\ref{tab:bounds}), so that the bound of Lemma~\ref{lem:fixedn} below is small.

\begin{lemma}\label{lem:transfer}
Let $j \geq 1$. If \eqref{eq:H1} is satisfied and $j\ge K_2-1$, or if \eqref{eq:H2} is satisfied and $|n_1|>K_1$, then
\begin{equation}\label{eq:contraction}
|M_j\zeta|_\nu\le\frac{1+\nu}2|\zeta|_\nu \quad \text{ for all } \zeta\in\C^2.
\end{equation}
\end{lemma}

\begin{proof}
Let $\zeta \in \C^2$ and let $\xi:=M_j\zeta$. Using \eqref{eq:rec} and $\gamma_{n_1}'-\gamma_{n_1}=d-1$, we get
\[
\xi^{\rm e}=-\zeta^{\rm e}+(\zeta^{\rm e}+\zeta^{\rm o}),\qquad\xi^{\rm e}+\xi^{\rm o}=\frac{2(d-1)j}{D_j}\,\zeta^{\rm e}+\Big(1-\frac{2j\gamma_{n_1}'}{D_j}\Big)(\zeta^{\rm e}+\zeta^{\rm o}).
\]
This implies that
\begin{equation}\label{eq:normMj}
|\xi|_\nu\le\max\Big\{\frac{1+\nu}2,\ \frac{4(d-1)j}{(\nu-1)|D_j|}+\Big|1-\frac{2j\gamma_{n_1}'}{D_j}\Big|\Big\}|\zeta|_\nu.
\end{equation}
Now, we have $-D_j = (j+a_{n_1})(j+b_{n_1})=P+iQ$ with $P:=(j+\frac\alpha2)(j+\frac{\alpha+1}2)-\frac{s_{n_1}^2}4$ and $Q:=s_{n_1}(j+\frac{2\alpha+1}4)$. Then $1-\frac{2j\gamma_{n_1}'}{D_j}=\frac{(P+2gj)+i(Q-2js_{n_1})}{P+iQ}$, so that, using $g+\alpha+\frac12=d-1$,
\begin{equation}\label{eq:modulus}
\Big|1-\frac{2j\gamma_{n_1}'}{D_j}\Big|^2=1+\frac{g\big(4j^3+4(d-1)j^2+\alpha(\alpha+1)j\big)-(d+\alpha-\frac12)js_{n_1}^2}{|D_j|^2}.
\end{equation}
If $j\ge K_2-1$, we use that $|D_j|\ge j^2$ by \eqref{eq:modD} and we drop the last term of the numerator, which is nonpositive. Since the left-hand side of \eqref{eq:H1} is decreasing in $K_2$, we obtain \eqref{eq:contraction} by combining \eqref{eq:normMj}, \eqref{eq:modulus} and \eqref{eq:H1}. If $|n_1| > K_1$, then $|n_1| \geq K_1+2$ since $n_1$ and $K_1$ are odd, hence $|s_{n_1}|\ge s_{K_1+2}$, and we use that $|D_j|\ge j^2+\frac{s_{n_1}^2}4\ge j|s_{n_1}|$. Moreover, with $t:=j/|s_{n_1}|$, we have
\begin{align*}
\frac{j^3}{|D_j|^2} &\leq \frac{1}{|s_{n_1}|}\frac{t^3}{(t^2+\frac14)^2}\le\frac{0.65}{s_{K_1+2}},\qquad
\frac{j^2}{|D_j|^2}\le \frac{1}{s_{n_1}^2}\frac{t^2}{(t^2+\frac14)^2}\le\frac{1}{s_{K_1+2}^2},\\
\frac{j}{|D_j|^2} &\leq \frac{1}{|s_{n_1}|^3}\frac{t}{(t^2+\frac14)^2}\le\frac{2.6}{s_{K_1+2}^3},
\end{align*}
where the maxima over $t>0$ are attained at $t^2=\frac34$, $\frac14$ and $\frac1{12}$ respectively. Since we also have $\frac{j}{|D_j|} \leq \frac{1}{s_{K_1+2}}$, we conclude the proof by combining \eqref{eq:normMj}, \eqref{eq:modulus} and \eqref{eq:H2}.
\end{proof}

Using the above control on the matrices $M_j$, we can estimate the entries of a given column of $\Lambda_{n_1}^{-1}$, and of $\partial_\omega\Lambda_{n_1}\Lambda_{n_1}^{-1}$, by the recursion \eqref{eq:rec}.

\begin{lemma}\label{lem:rows}
Let $k >0$,   $0\le J<k$ and assume that \eqref{eq:contraction} is satisfied for all $J<j<k$. Then, recalling the definition of $\beta_{n_1}$ in \eqref{eq:beta},
\begin{align}
\sum_{j=J}^k\nu^{j-k}\big|\big(\Lambda_{n_1}^{-1}\big)_{jk}\big|&\le\frac1{4|D_k|}\Big(1+\frac{k\beta_{n_1}}{|D_J|}\Big),\label{eq:rows1}\\
\sum_{j=J}^k\nu^{j-k}\big|\big(\partial_\omega\Lambda_{n_1}\Lambda_{n_1}^{-1}\big)_{jk}\big|&\le\frac{|n_1|}{2|D_k|}\Big(|2k+\alpha+\tfrac12+is_{n_1}|+\frac{2k\beta_{n_1}}{|J+a_{n_1}|}+\frac{8k}{\nu-1}\Big).\label{eq:rows2}
\end{align}
\end{lemma}

\begin{proof}
Let $x:=H_{n_1}^{-1}e_k$, so that $\Lambda_{n_1}^{-1}e_k=-\frac14x$ and $x_k=\frac{1}{D_k}$ by Lemma~\ref{lem:rec}. Since $\zeta_{k-1}=(0,\frac{2k}{D_k})^\top$ and $\zeta_{j-1}=M_j\zeta_j$, \eqref{eq:contraction} yields
\begin{equation}\label{eq:state}
|\zeta_j|_\nu\le\frac{4k}{(\nu-1)|D_k|}\Big(\frac{1+\nu}2\Big)^{k-1-j} \quad \text{ for all } J\le j<k.
\end{equation}
Now, using \eqref{eq:rec}, we have $x_j=-\frac{\mu_j}{D_j}\big(-(d-1)\zeta_j^{\rm e}+\gamma_{n_1}'(\zeta_j^{\rm e}+\zeta_j^{\rm o})\big)$, hence
\begin{equation}\label{eq:xj}
|x_j|\le\frac{2(d-1)+(\nu-1)|\gamma_{n_1}'|}{2|D_j|}|\zeta_j|_\nu \quad \text{ for all } J\le j<k.
\end{equation}
Since $|D_j|$ is increasing in $j$ and $\sum_{m\ge1}\nu^{-m}\big(\frac{1+\nu}2\big)^{m-1}=\frac2{\nu-1}$, we obtain \eqref{eq:rows1} by combining \eqref{eq:state}, \eqref{eq:xj} and \eqref{eq:beta}. 

We now turn to \eqref{eq:rows2}. Let $\sigma_j:=2j+\alpha+\frac12+is_{n_1}=(j+a_{n_1})+(j+b_{n_1})$. Using \eqref{eq:dLam} and the definition of $\zeta_j$ given in the proof of Lemma~\ref{lem:rec}, we have for all $j<k$
\begin{align*}
(\partial_\omega\Lambda_{n_1}x)_j&=2in_1\sigma_jx_j+4in_1\mu_j(\zeta_j^{\rm e}+\zeta_j^{\rm o})\\
&=4in_1\mu_j\Big[\frac{(d-1)\sigma_j}{2D_j}\zeta_j^{\rm e}+\Big(1-\frac{\sigma_j\gamma_{n_1}'}{2D_j}\Big)(\zeta_j^{\rm e}+\zeta_j^{\rm o})\Big].
\end{align*}
Since $|j+a_{n_1}| \leq |j+b_{n_1}|$, we have $\frac{|\sigma_j|}{|D_j|}\le \frac{2}{|j+a_{n_1}|}$, hence
\[
\big|(\partial_\omega\Lambda_{n_1}x)_j\big|\le2|n_1|\Big(\frac{2(d-1)+(\nu-1)|\gamma_{n_1}'|}{|j+a_{n_1}|}+\nu-1\Big)|\zeta_j|_\nu \quad \text{ for all } J\le j<k.
\]
We conclude the proof using \eqref{eq:state}, the identity $(\partial_\omega\Lambda_{n_1}\Lambda_{n_1}^{-1})_{kk}=-\frac{in_1\sigma_k}{2D_k}$ and that $|j+a_{n_1}|$ is increasing in $j$.
\end{proof}

We now turn to the columns with $|n_1| > K_1$, which are bounded uniformly in the next lemma.
\begin{lemma}\label{lem:T1}
  Assume that \eqref{eq:H2} is satisfied and let $|n_1| > K_1$. Then 
  \begin{align*}
    \|\Lambda_{n_1}^{-1}\|_{\ell^1_\nu} \leq \frac{\eta}{|n_1|^2} \quad  \text{ and } \quad  \|\partial_\omega\Lambda_{n_1}\Lambda_{n_1}^{-1}\|_{\ell^1_\nu} \leq \eta',
  \end{align*}
  where, recalling \eqref{eq:beta},
  \begin{equation}\label{eq:eta}
    \eta := \frac1{\bar\omega^2}\Big(1+\frac{\beta_{K_1+2}}{(K_1+2)\bar\omega}\Big) \quad \text{ and } \quad 
\eta':=\frac1{\bar\omega}\Big(2+\frac4{\nu-1}+\frac{2\beta_{K_1+2}}{(K_1+2)\bar\omega}\Big).
  \end{equation}
\end{lemma}

\begin{proof}
Let $k \geq 0$. By Lemma~\ref{lem:transfer}, \eqref{eq:contraction} holds for all $j \geq 1$. Consequently, we can apply Lemma~\ref{lem:rows} with $J=0$ if $k \geq 1$, and \eqref{eq:rows1} and \eqref{eq:rows2} trivially hold if $k=0$. Now, using \eqref{eq:modD}, we have $|D_k|\ge k^2+\frac{s_{n_1}^2}4\ge \max\{k|s_{n_1}|, \frac{s_{n_1}^2}{4}\}$ and $|D_0|\ge\frac{s_{n_1}^2}4$. Then, \eqref{eq:rows1} provides
\begin{align*}
  \|\Lambda_{n_1}^{-1}\pi_k\|_{\ell^1_\nu} \leq \frac1{4|D_k|} + \frac{k\beta_{n_1}}{4|D_k||D_0|} \leq \frac{1}{s_{n_1}^2}\Big(1 + \frac{\beta_{n_1}}{|s_{n_1}|}\Big).
\end{align*}
Similarly, using that $|a_{n_1}|\ge\frac{|s_{n_1}|}2$ and $\frac{|2k+\alpha+\frac12+is_{n_1}|}{|D_k|}\le\frac{2}{|k+a_{n_1}|} \leq \frac{4}{|s_{n_1}|}$, \eqref{eq:rows2} provides
\begin{align*}
  \|\partial_\omega\Lambda_{n_1}\Lambda_{n_1}^{-1}\pi_k\|_{\ell^1_\nu} \leq \frac{|n_1|}{|s_{n_1}|}\Big(2 + \frac{2\beta_{n_1}}{|s_{n_1}|} + \frac{4}{\nu-1}\Big).
\end{align*}
We conclude the proof using that $s_{n_1} = n_1 \bar\omega$, that $|n_1| \geq K_1+2$, and that $\frac{\beta_{n_1}}{|s_{n_1}|}$ is decreasing in $|n_1|$ by \eqref{eq:modD} and \eqref{eq:beta}.
\end{proof}

Finally, we estimate the norm of the columns satisfying $|n_1| \leq  K_1$ and $n_2 > K_2$. The idea is that the entries $j \leq K_2-2$ of such a column are linear combinations of the corresponding entries of the columns $K_2-1$ and $K_2$, which are controlled by \eqref{eq: finite number of columns} and \eqref{eq: finite number of columns derivative}. The remaining entries are controlled by Lemma~\ref{lem:rows}.
\begin{lemma}\label{lem:fixedn}
  Assume that \eqref{eq:H1} is satisfied and that $\eps_{n_1} < 1$ for all $|n_1| \leq K_1$. Moreover, let 
  \begin{equation}\label{eq:constants}
\begin{aligned}
\theta_1&:=\frac{(1+\nu)^2}{\nu^2(\nu-1)(K_2+1)}\Big(|\gamma_{K_1}'|+\frac{(1+\nu)(4K_2^2+K_1^2\bar\omega^2)}{16(K_2-1)}\Big),\\
\theta_2&:=\frac{(1+\nu)^2\big(4(K_2+1)^2+K_1^2\bar\omega^2\big)}{8\nu(\nu-1)K_2(K_2+1)},\\
\theta_3&:=\frac1{4(K_2-1)^2}\Big(1+\frac{\beta_{K_1}}{K_2+1}\Big),\qquad
\theta_3':=\frac1{K_2+1}\Big(1+\frac4{\nu-1}+\frac{\beta_{K_1}}{K_2-1}\Big).
\end{aligned}
\end{equation}
Then, recalling the definitions of  $\lambda_n$ and $\lambda'_n$ in \eqref{eq: finite number of columns} and \eqref{eq: finite number of columns derivative} respectively, for all $n = (n_1,n_2)$ with $|n_1| \leq K_1$ and $n_2 > K_2$, we have
\begin{align*}
  \|\Lambda^{-1}\pi_n\|_{\ell^1} &\leq \theta_1 \lambda_{(n_1,K_2-1)} + \theta_2 \lambda_{(n_1,K_2)} + \theta_3,\\
   \|\partial_\omega\Lambda\,\Lambda^{-1}\pi_n\|_{\ell^1} &\leq \theta_1 \lambda_{(n_1,K_2-1)}' + \theta_2 \lambda_{(n_1,K_2)}' + |n_1|\theta_3'.
\end{align*}
\end{lemma}

\begin{proof}
Let $k := n_2 > K_2$ and let $B$ denote either $\Lambda_{n_1}^{-1}$ or $\partial_\omega\Lambda_{n_1}\Lambda_{n_1}^{-1}$. We need to estimate $\|B \pi_k\|_{\ell^1_\nu}$. Given $l \geq K_2-1$, let $x^{(l)} := H_{n_1}^{-1}e_l$ and let $\zeta^{(l)}_j$ be the associated vectors given in Lemma~\ref{lem:rec}. Using \eqref{eq:rec} and the proof of Lemma~\ref{lem:rows}, the entries $B_{jl}$, $j \leq K_2-2$, are obtained by applying to $\zeta^{(l)}_{K_2-2}$ a linear map which does not depend on $l$. Now, Lemma~\ref{lem:rec} provides
\[
\zeta^{(K_2-1)}_{K_2-2} = \Big(0,\frac{2(K_2-1)}{D_{K_2-1}}\Big)^{\!\top} \quad \text{ and } \quad \zeta^{(K_2)}_{K_2-2} = \frac{2K_2}{D_{K_2}}\Big(1,-\frac{2(K_2-1)\gamma_{n_1}'}{D_{K_2-1}}\Big)^{\!\top}.
\]
Decomposing $\zeta^{(k)}_{K_2-2}=:(\zeta^{\rm e},\zeta^{\rm o})^\top$ in this basis of $\C^2$, we obtain that for all $j\le K_2-2$
\begin{equation}\label{eq:twocol}
B_{jk}=\Big[\Big(\gamma_{n_1}'-\frac{D_{K_2-1}}{2(K_2-1)}\Big)\zeta^{\rm e}+\frac{D_{K_2-1}}{2(K_2-1)}(\zeta^{\rm e}+\zeta^{\rm o})\Big]B_{j,K_2-1}+\frac{D_{K_2}}{2K_2}\,\zeta^{\rm e}\,B_{j,K_2}.
\end{equation}
By Lemma~\ref{lem:transfer}, \eqref{eq:contraction} holds for all $j \geq K_2-1$, and therefore \eqref{eq:state} yields
\[
|\zeta^{(k)}_{K_2-2}|_\nu\le\frac{4k}{(\nu-1)|D_k|}\Big(\frac{1+\nu}2\Big)^{k-K_2+1} \leq \frac{4}{(\nu-1)(K_2+1)}\Big(\frac{1+\nu}2\Big)^{k-K_2+1},
\]
where we used that $|D_k| \geq k^2$ and $k \geq K_2+1$. Moreover, since $\alpha\le1$, \eqref{eq:modD} gives $|D_j|\le(j+1)^2+\frac{s_{n_1}^2}4$, with $|s_{n_1}|\le K_1\bar\omega$. Then, multiplying \eqref{eq:twocol} by $\nu^{j-k}$, summing over $j\le K_2-2$ and using that $(\frac{1+\nu}{2\nu})^{k-K_2+1} \leq (\frac{1+\nu}{2\nu})^{2}$, we get
\begin{equation}\label{eq:lowrows}
  \sum_{j=0}^{K_2-2}\nu^{j-k}|B_{jk}| \leq \theta_1\|B\pi_{K_2-1}\|_{\ell^1_\nu}+\theta_2\|B\pi_{K_2}\|_{\ell^1_\nu}.
\end{equation}
The norms $\|B\pi_{K_2-1}\|_{\ell^1_\nu}$ and $\|B\pi_{K_2}\|_{\ell^1_\nu}$ are bounded by \eqref{eq: finite number of columns} if $B = \Lambda_{n_1}^{-1}$, and by \eqref{eq: finite number of columns derivative} if $B = \partial_\omega\Lambda_{n_1}\Lambda_{n_1}^{-1}$. 

It remains to estimate the entries $B_{jk}$ with $K_2-1 \leq j \leq k$. For this purpose, we apply Lemma~\ref{lem:rows} with $J=K_2-1$. Using that $|D_k|\ge k^2 \geq (K_2+1)^2$, $|D_{K_2-1}| \geq (K_2-1)^2$, $|K_2-1+a_{n_1}|\ge K_2-1$, $\frac{|2k+\alpha+\frac12+is_{n_1}|}{|D_k|} \leq \frac{2}{k}$ and $\beta_{n_1} \leq \beta_{K_1}$, \eqref{eq:rows1} and \eqref{eq:rows2} imply that
\begin{equation}\label{eq:highrows}
  \sum_{j=K_2-1}^{k}\nu^{j-k}\big|\big(\Lambda_{n_1}^{-1}\big)_{jk}\big| \leq \theta_3 \quad \text{ and } \quad \sum_{j=K_2-1}^{k}\nu^{j-k}\big|\big(\partial_\omega\Lambda_{n_1}\Lambda_{n_1}^{-1}\big)_{jk}\big| \leq |n_1|\theta_3'.
\end{equation}
We conclude the proof by combining \eqref{eq:lowrows} and \eqref{eq:highrows}.
\end{proof}

Combining \eqref{eq: finite number of columns} with Lemmas~\ref{lem:T1} and~\ref{lem:fixedn}, we control all the columns of $\Lambda^{-1}$, which provides the desired bound on the norm of $\Lambda^{-1}\pi^{>N}$.

\begin{lemma}\label{lem:kappa}
Assume that \eqref{eq:H1} and \eqref{eq:H2} are satisfied and that $\eps_{n_1} < 1$ for all $|n_1| \leq K_1$. Moreover, let $\kappa$ be defined as 
\begin{equation}\label{eq:kappa}
  \kappa := \max\Big\{ \max_{n \leq K,\, n \not\leq N} \lambda_n, ~ \max_{|n_1| \leq K_1} \big(\theta_1 \lambda_{(n_1,K_2-1)} + \theta_2 \lambda_{(n_1,K_2)} + \theta_3\big), ~ \frac{\eta}{(K_1+2)^2}\Big\},
\end{equation}
where the maxima are taken over odd values of $n_1$. Then, $\Lambda^{-1}$ is a bounded linear operator on $\ell^1$ and $\|\Lambda^{-1}\pi^{>N}\|_{\ell^1} \leq \kappa$. 
\end{lemma}

\begin{proof}
Using \eqref{eq:opnorm}, we have $\|\Lambda^{-1}\pi^{>N}\|_{\ell^1} = \sup_{n \not\leq N} \|\Lambda^{-1}\pi_n\|_{\ell^1}$. If $n \leq K$, then $\|\Lambda^{-1}\pi_n\|_{\ell^1} \leq \lambda_n$ by \eqref{eq: finite number of columns}. If $|n_1| \leq K_1$ and $n_2 > K_2$, we use Lemma~\ref{lem:fixedn}. Finally, if $|n_1| > K_1$, then $|n_1| \geq K_1 + 2$ and we use Lemma~\ref{lem:T1}. The same argument shows that the norms of all the columns of $\Lambda^{-1}$ are uniformly bounded, that is, $\Lambda^{-1}$ is bounded on $\ell^1$. 
\end{proof}

Note that $\Lambda_{-n_1} = \ov{\Lambda_{n_1}}$, hence the quantities $\lambda_n$ only need to be computed for $n_1 > 0$. Thus the bound $\kappa$ in \eqref{eq:kappa} requires the computation of $\frac{K_1+1}{2}$ matrices of size $K_2+1$.

\subsubsection{Columns computed exactly}\label{sec:exact} By \eqref{eq:Z12product}, the contribution of the column $n\not\le N$ to $Z_{1,2}$ is at most $\big(\|p\bar W^{p-1}\|_{\ell^1}+\frac1{2\nu_\theta}\big)\|\Lambda^{-1}\pi_n\|_{\ell^1}$, where $\|p\bar W^{p-1}\|_{\ell^1}=\|\bar V\|_{\ell^1}+\|p\bar W^{p-1}-\bar V\|_{\ell^1}$ is large (between $1.3\cdot10^5$ and $4\cdot10^5$ for the three cases of \eqref{eq:cases}). This bound is far from sharp near the turning points of the recursion of Lemma~\ref{lem:rec} (see Remark~\ref{rem:cost}), where the column $\Lambda_{n_1}^{-1}e_{n_2}$ and the Fourier modes of $p\bar W^{p-1}$ oscillate coherently, so that the norm of their product is much smaller than the product of their norms. For finitely many columns $n\le K$, we therefore compute the quotients of Lemma~\ref{lem:Z1} directly. Given $m_E\in\N$, we write $\bar V=\bar V_{\le m_E}+\bar V_{>m_E}$, where $\bar V_{\le m_E}$ contains the Fourier modes $|m_1|\le m_E$ of $\bar V$. For $n\le K$, \eqref{eq:tri} shows that $\Lambda^{-1}e_n$ is supported on the Fourier mode $n_1$, where it coincides with $\big(\Lambda_{n_1}^{(K_2)}\big)^{-1}e_{n_2}$, and the Neumann series argument of \eqref{eq: finite number of columns} gives $\big\|\big(\Lambda_{n_1}^{(K_2)}\big)^{-1}-X_{n_1}\big\|_{\ell^1_\nu}\le\frac{\eps_{n_1}}{1-\eps_{n_1}}\|X_{n_1}\|_{\ell^1_\nu}$. Moreover, $\ell(\Lambda^{-1}e_n)=0$ unless $|n_1|=1$. Consequently, for all $n\le K$,
\begin{multline}\label{eq:colMV}
\frac{|\ell(\Lambda^{-1}e_n)|+\|M_{p\bar W^{p-1}}\Lambda^{-1}e_n\|_{\ell^1}}{\|e_n\|_{\ell^1}}\le\frac{\|M_{\bar V_{\le m_E}}X_{n_1}e_{n_2}\|_{\ell^1}}{\|e_n\|_{\ell^1}}\\
+\Big(\|\bar V_{>m_E}\|_{\ell^1}+\|p\bar W^{p-1}-\bar V\|_{\ell^1}+\frac{\delta_{|n_1|,1}}{2\nu_\theta}\Big)\lambda_n+\|p\bar W^{p-1}\|_{\ell^1}\frac{\eps_{n_1}\|X_{n_1}\|_{\ell^1_\nu}}{1-\eps_{n_1}},
\end{multline}
where $X_{n_1}e_{n_2}$ is identified with the sequence supported on the Fourier mode $n_1$, and $\delta_{|n_1|,1}$ is the Kronecker symbol. The first term of the right-hand side is the norm of a finite sequence, the product of the polynomial $X_{n_1}e_{n_2}$ with the Fourier modes $|m_1|\le m_E$ of $\bar V$, which we compute with interval arithmetic by the fast Fourier transform. Given a threshold $\vartheta>0$, we use \eqref{eq:colMV} for the columns $n\not\le N$, $n\le K$, such that $\big(\|p\bar W^{p-1}\|_{\ell^1}+\frac1{2\nu_\theta}\big)\lambda_n>\vartheta$, and the bound \eqref{eq:Z12product}, column by column, for all the other columns $n\not\le N$. Following the proof of Lemma~\ref{lem:kappa}, we may then take for $Z_{1,2}$ the maximum of the right-hand sides of \eqref{eq:colMV} over the first set of columns and of $\big(\|p\bar W^{p-1}\|_{\ell^1}+\frac1{2\nu_\theta}\big)$ times the right-hand side of \eqref{eq:kappa}, in which the first maximum is restricted to the other columns. As for $\lambda_n$, the quotients only need to be computed for $n_1>0$: since $\bar W$ is real, the conjugation $\mathcal C$ of Remark~\ref{rem : the solution is real} maps the column $(n_1,n_2)$ to the column $(-n_1,n_2)$ and preserves the norms.

\subsection{Computation of the $Z_2$ bound}\label{sec:Z2}

In this section, we compute the bound $Z_2$ of Proposition~\ref{th : radii polynomial}. Since $F$ is polynomial, this bound follows from the Banach algebra property of $\ell^1_{\nu_\theta,\nu}$ and from bounds on the norms of the operators $\pi_\omega A$, $\pi_W A$, $\partial_\omega\Lambda\,\pi_W A$ and $\mathcal N^2\pi_WA$.

\begin{lemma}\label{lem:Z2}
Let $r > 0$ and let $Z_2(r)$ be a bound satisfying
\begin{align*}
  Z_2(r) \geq~ &p\sum_{j=1}^{p-1}\binom{p-1}{j}\|\bar W^{p-1-j}\|_{\ell^1}\|\pi_WA\|_{\mathcal{B}(X)}^{j+1}\,r^{j-1}+2\,\|\pi_\omega A\|_{\mathcal{B}(X)}\|\partial_\omega\Lambda\,\pi_W A\|_{\mathcal{B}(X)}\\
  &+\|\pi_\omega A\|_{\mathcal{B}(X)}^2\big(2\,\|\mathcal N^2\bar W\|_{\ell^1}+3\,\|\mathcal N^2\pi_WA\|_{\mathcal{B}(X)}\,r\big),
\end{align*}
with the convention $\|\bar W^0\|_{\ell^1}=1$. Then, $\|\left({D}{F}(\bar{U} + A h) - D{F}(\bar{U})\right)A\|_{\mathcal{B}(X)} \leq Z_2(r)r$ for all $h \in B_r(0) \subset X$.
\end{lemma}

\begin{proof}
Let $h \in B_r(0) \subset X$ and let $\Phi\in X$. Moreover, let us denote $(\varsigma,w):=Ah$ and $(\varsigma',w'):=A\Phi$, so that $\bar U+Ah=(\bar\omega+\varsigma,\bar W+w)$. Using \eqref{eq:DF} and \eqref{eq:dLam}, the first component of $\left({D}{F}(\bar{U} + A h) - D{F}(\bar{U})\right)A\Phi$ vanishes, and its second component is given by
\begin{equation}\label{eq:second derivative}
\varsigma\,\partial_\omega\Lambda w'-\varsigma^2\mathcal N^2w'+p\big((\bar W+w)^{p-1}-\bar W^{p-1}\big)*w'+\varsigma'\big(\partial_\omega\Lambda w-2\varsigma\,\mathcal N^2(\bar W+w)\big).
\end{equation}
Now, by definition of the operator norm, we have
\begin{align*}
|\varsigma|&\le \|\pi_\omega A\|_{\mathcal{B}(X)}\|h\|_X,& \|w\|_{\ell^1}&\le \|\pi_W A\|_{\mathcal{B}(X)}\|h\|_X,\\
 \|\partial_\omega\Lambda w\|_{\ell^1}&\le \|\partial_\omega\Lambda\,\pi_W A\|_{\mathcal{B}(X)}\|h\|_X,& \|\mathcal N^2w\|_{\ell^1}&\le \|\mathcal N^2\pi_W A\|_{\mathcal{B}(X)}\|h\|_X,
\end{align*}
and the same estimates hold for $(\varsigma',w')$ and $\Phi$. Moreover, writing $(\bar W+w)^{p-1}-\bar W^{p-1}=\sum_{j=1}^{p-1}\binom{p-1}{j}\bar W^{p-1-j}*w^j$, where $\bar W^0*w^{p-1}=w^{p-1}$, and using the Banach algebra property of $\ell^1_{\nu_\theta,\nu}$, we get
\[
\|p\big((\bar W+w)^{p-1}-\bar W^{p-1}\big)\|_{\ell^1}\le p\sum_{j=1}^{p-1}\binom{p-1}{j}\|\bar W^{p-1-j}\|_{\ell^1}\|w\|_{\ell^1}^j.
\]
Consequently, using that $\|h\|_X \leq r$, the $\ell^1$ norm of \eqref{eq:second derivative} is bounded by $Z_2(r)\|h\|_X\|\Phi\|_X \leq Z_2(r)r\|\Phi\|_X$.
\end{proof}

The norms $\|\bar W^m\|_{\ell^1}$, $1\le m\le p-2$, are norms of finite sequences, computed as for \eqref{eq:Ybound}; they can be much smaller than $\|\bar W\|_{\ell^1}^m$. In order to apply the above lemma, it remains to compute the norm of the operators $\pi_\omega A$, $\pi_W A$, $\partial_\omega\Lambda\,\pi_W A$ and $\mathcal N^2\pi_WA$. Using \eqref{eq:A} and \eqref{eq:opnorm}, we have
\begin{equation}\label{eq:normsA}
\begin{aligned}
  \|\pi_\omega A\|_{\mathcal{B}(X)} &= \|\pi_\omega A_N \pi^{\leq N}\|_{\mathcal{B}(X)},\\
  \|\pi_W A\|_{\mathcal{B}(X)} &= \max\big\{\|\pi_W A_N \pi^{\leq N}\|_{\mathcal{B}(X)}, ~ \|\Lambda^{-1}\pi^{>N}\|_{\ell^1}\big\},\\
  \|\partial_\omega\Lambda\,\pi_W A\|_{\mathcal{B}(X)} &= \max\big\{\|\partial_\omega\Lambda\,\pi_W A_N \pi^{\leq N}\|_{\mathcal{B}(X)}, ~ \|\partial_\omega\Lambda\,\Lambda^{-1}\pi^{>N}\|_{\ell^1}\big\},\\
  \|\mathcal N^2\pi_W A\|_{\mathcal{B}(X)} &= \max\big\{\|\mathcal N^2\pi_W A_N \pi^{\leq N}\|_{\mathcal{B}(X)}, ~ \|\mathcal N^2\Lambda^{-1}\pi^{>N}\|_{\ell^1}\big\}.
\end{aligned}
\end{equation}
By \eqref{eq:tri}, the quantities involving $A_N$ in \eqref{eq:normsA} are matrix norms, computed with interval arithmetic. Moreover, $\|\Lambda^{-1}\pi^{>N}\|_{\ell^1}$ is controlled by Lemma~\ref{lem:kappa}. The two remaining quantities are obtained similarly, using the analysis of Section~\ref{sec:main}. Specifically, we have the following result.

\begin{lemma}\label{lem:kappa prime}
Assume that \eqref{eq:H1} and \eqref{eq:H2} are satisfied and that $\eps_{n_1} < 1$ for all $|n_1| \leq K_1$. Moreover, let $\kappa'$ and $\kappa''$ be defined as 
\begin{equation}\label{eq:kappa prime}
\begin{aligned}
  \kappa' &:= \max\Big\{ \max_{n \leq K,\, n \not\leq N} \lambda_n', ~ \max_{|n_1| \leq K_1} \big(\theta_1 \lambda_{(n_1,K_2-1)}' + \theta_2 \lambda_{(n_1,K_2)}' + |n_1|\theta_3'\big), ~ \eta'\Big\},\\
  \kappa'' &:= \max\Big\{ \max_{n \leq K,\, n \not\leq N} n_1^2\lambda_n, ~ \max_{|n_1| \leq K_1} n_1^2\big(\theta_1 \lambda_{(n_1,K_2-1)} + \theta_2 \lambda_{(n_1,K_2)} + \theta_3\big), ~ \eta\Big\},
\end{aligned}
\end{equation}
where the maxima are taken over odd values of $n_1$. Then, $\partial_\omega\Lambda\,\Lambda^{-1}$ and $\mathcal N^2\Lambda^{-1}$ are bounded linear operators on $\ell^1$. Moreover, $\|\partial_\omega\Lambda\,\Lambda^{-1}\pi^{>N}\|_{\ell^1} \leq \kappa'$ and  $\|\mathcal N^2\Lambda^{-1}\pi^{>N}\|_{\ell^1} \leq \kappa''$.
\end{lemma}

\begin{proof}
The proof for $\partial_\omega\Lambda\,\Lambda^{-1}$ is the same as that of Lemma~\ref{lem:kappa}, using \eqref{eq: finite number of columns derivative} and the second estimates of Lemmas~\ref{lem:T1} and~\ref{lem:fixedn}. For $\mathcal N^2\Lambda^{-1}$, we use that $\mathcal N^2\Lambda^{-1}\pi_n=n_1^2\Lambda^{-1}\pi_n$ for all $n$, combined with the proof of Lemma~\ref{lem:kappa}.
\end{proof}

\subsection{Proof of Proposition~\ref{prop:profiles}}\label{sec:result}

In this section, we combine the results of the previous sections in order to apply Proposition~\ref{th : radii polynomial} and prove Proposition~\ref{prop:profiles}. We first prove that $A$ is injective if $Z_{1}<1$.

\begin{lemma}\label{lem:setting}
Under the assumptions of Lemma~\ref{lem:kappa}, let $Z_{1}$ be given by Lemma~\ref{lem:Z1}. If $Z_{1} < 1$, then $A$, given in \eqref{eq:A}, is injective on $X$.
\end{lemma}

\begin{proof}
If $Z_1 < 1$, then using a Neumann series argument, $DF(\bar U) A : X \to X$ is invertible, hence $A$ is injective on $X$.
\end{proof}

We are now in a position to state the computer-assisted part of the proof of Proposition~\ref{prop:profiles}. For each of the pairs $(d,p)$ in \eqref{eq:cases}, the approximate solution $\bar U = (\bar\omega, \bar W)$ is the one stored in \cite{code}, obtained as described in Section~\ref{sec:profiles}. It is real, $\bar W$ is supported on the Fourier modes $|n_1| \leq N_{W,1}$ and the Chebyshev modes $n_2 \leq N_{W,2}$, and
\begin{equation}\label{eq:omega bar}
  |\bar\omega - \omega_0| \leq 10^{-15},
\end{equation}
where $N_W = (N_{W,1},N_{W,2})$ and $\omega_0$ are given in Table~\ref{tab:proof}.

\begin{table}[t]
\centering
\begin{tabular}{|c|c|c|c|}
\hline
$(d,p)$ & $(9,3)$ & $(7,5)$ & $(6,7)$\\
\hline
$N_W$ & $(103,1290)$ & $(101,600)$ & $(145,600)$\\
$\omega_0$ & $40.260236233935841$ & $24.068215567617913$ & $19.908444486656099$\\
$N$ & $(65,1400)$ & $(65,900)$ & $(95,1000)$\\
$r_0$ & $2.7\cdot10^{-18}$ & $3.6\cdot10^{-14}$ & $3.6\cdot10^{-12}$\\
$r_1$ & $5.4\cdot10^{-13}$ & $3.1\cdot10^{-10}$ & $4.7\cdot10^{-10}$\\
$\delta_\omega$ & $1.39\cdot10^{-16}$ & $4.93\cdot10^{-13}$ & $3.12\cdot10^{-11}$\\
$\delta_W$ & $2.30\cdot10^{-14}$ & $3.71\cdot10^{-12}$ & $1.25\cdot10^{-10}$\\
\hline
\end{tabular}
\medskip
\caption{Approximate solutions and parameters of Proposition~\ref{th:proof}.}
\label{tab:proof}
\end{table}

\begin{proposition}\label{th:proof}
Let $(d,p)$ be one of the pairs in \eqref{eq:cases}, let $\nu = 1.1$ and $\nu_\theta = 1.01$, and let $N$, $r_0$, $r_1$, $\delta_\omega$ and $\delta_W$ be as in Table~\ref{tab:proof}. Then, there exists a finite-rank operator $A_N = \pi^{\leq N}A_N \pi^{\leq N}$ such that, with $A$ given in \eqref{eq:A}, the following holds. For all $r \in [r_0, r_1]$, there exists a unique $\tilde\Psi \in B_r(0) \subset X$ such that $\tilde U = (\tilde \omega, \tilde W) = \bar U + A\tilde\Psi$ satisfies $F(\tilde U) = 0$. Moreover, $\tilde U$ is real and
\begin{equation}\label{eq:enclosure}
  |\tilde\omega - \bar\omega| \leq \delta_\omega \quad \text{ and } \quad \|\tilde W - \bar W\|_{\ell^1} \leq \delta_W.
\end{equation}
\end{proposition}

\begin{table}[t]
\centering
\begin{tabular}{|c|c|c|c|}
\hline
$(d,p)$ & $(9,3)$ & $(7,5)$ & $(6,7)$\\
\hline
$K$ & $(165,10500)$ & $(203,7000)$ & $(205,8000)$\\
$\|\bar W\|_{\ell^1}$ & $842.35$ & $35.256$ & $16.833$\\
$\|p\bar W^{p-1}\|_{\ell^1}$ & $3.9733\cdot10^{5}$ & $1.3362\cdot10^{5}$ & $2.2907\cdot10^{5}$\\
$\|\Lambda^{-1}\pi^{>N}\|_{\ell^1}$ & $2.134\cdot10^{-6}$ & $6.432\cdot10^{-6}$ & $4.002\cdot10^{-6}$\\
$Y$ & $5.39\cdot10^{-19}$ & $7.23\cdot10^{-15}$ & $6.29\cdot10^{-13}$\\
$Z_{1,0}$ & $1.81\cdot10^{-4}$ & $1.62\cdot10^{-7}$ & $5.49\cdot10^{-8}$\\
$Z_{1,1}$ & $0.3443$ & $0.2711$ & $0.3586$\\
$Z_{1,2}$ & $0.7995$ & $0.7987$ & $0.8000$\\
$\|\pi_\omega A\|_{\mathcal B(X)}$ & $51.20$ & $13.69$ & $8.663$\\
$\|\pi_W A\|_{\mathcal B(X)}$ & $8510$ & $102.8$ & $34.54$\\
$Z_2(r)$, $r\le10^{-6}$ & $3.674\cdot10^{11}$ & $6.434\cdot10^{8}$ & $4.219\cdot10^{8}$\\
\hline
\end{tabular}
\medskip
\caption{The main bounds of the proof of Proposition~\ref{th:proof}. Each entry is an upper bound, except $K$, which is a parameter, and $\|\bar W\|_{\ell^1}$, which is given to five significant digits.}
\label{tab:bounds}
\end{table}

\begin{proof}
We choose $K$ as in Table~\ref{tab:bounds} and $\bar V := \pi^{\leq N + (10,10)}(p\bar W^{p-1})$. Since $\bar U$ is real, the operator $\pi^{\leq N}DF(\bar U)\pi^{\leq N}$ maps the real elements of $\pi^{\leq N}X$ to real elements, on which it is a real matrix of size $(N_1+1)(N_2+1)+1$. Let $B_N$ be the floating-point approximate inverse of this matrix computed in \cite{code}, and let $A_N := \frac{1}{2}(B_N + \mathcal C B_N \mathcal{C})$ be the complex linear operator on $\pi^{\leq N}X$ which coincides with $B_N$ on the real elements. Then $\mathcal{C}A = A\mathcal{C}$, where $\mathcal{C}$ is defined in Remark~\ref{rem : the solution is real}. All the quantities below are computed in \cite{code} using interval arithmetic \cite{julia_interval, julia_olivier}, and the coefficients of $\bar W^{p-1}$ and $\bar W^p$ are computed by the fast Fourier transform, with $192$ bits of precision.

First, we verify that \eqref{eq:H1} and \eqref{eq:H2} are satisfied, and we construct the matrices $X_{n_1}$ of Section~\ref{sec:main}, which satisfy $\eps_{n_1} \leq 5\cdot 10^{-11}$ for all $|n_1| \leq K_1$. Then, Lemma~\ref{lem:kappa} yields the bound on $\|\Lambda^{-1}\pi^{>N}\|_{\ell^1}$ of Table~\ref{tab:bounds}, and Lemma~\ref{lem:kappa prime} provides bounds on $\|\partial_\omega\Lambda\,\Lambda^{-1}\pi^{>N}\|_{\ell^1}$ and $\|\mathcal N^2\Lambda^{-1}\pi^{>N}\|_{\ell^1}$. Using \eqref{eq:Ybound} and Lemma~\ref{lem:Z1}, where $Z_{1,2}$ is computed as explained in Section~\ref{sec:exact} with $\vartheta = 0.8$ and $m_E = 30$, we can choose $Y$, $Z_{1,0}$, $Z_{1,1}$ and $Z_{1,2}$ as in Table~\ref{tab:bounds}, so that
\[
  Z_1 = \max\{Z_{1,0}+Z_{1,1},\, Z_{1,2}\} = Z_{1,2} < 1.
\]
In particular, $A$ is injective by Lemma~\ref{lem:setting}. Then, \eqref{eq:normsA} provides the bounds on $\|\pi_\omega A\|_{\mathcal{B}(X)}$ and $\|\pi_W A\|_{\mathcal{B}(X)}$ of Table~\ref{tab:bounds}, as well as bounds on $\|\partial_\omega\Lambda\,\pi_W A\|_{\mathcal{B}(X)}$ and $\|\mathcal N^2\pi_W A\|_{\mathcal{B}(X)}$, and Lemma~\ref{lem:Z2} allows us to choose $Z_2$ non-decreasing such that $Z_2(r)$ is given in Table~\ref{tab:bounds} for all $r \leq 10^{-6}$. With these values, one readily verifies that \eqref{condition radii polynomial} is satisfied for all $r \in [r_0, r_1]$: the radii polynomial is convex and negative at $r_0$ and at $r_1$, and $Z_1+Z_2r_1<1$. Proposition~\ref{th : radii polynomial} then provides the existence and the uniqueness of $\tilde \Psi$. Since $\bar U$ is real and $\mathcal{C}A = A\mathcal{C}$, Remark~\ref{rem : the solution is real} implies that $\tilde U$ is real. Finally, since $\|\tilde \Psi\|_X \leq r_0$, we obtain
\[
  |\tilde\omega - \bar\omega| \leq \|\pi_\omega A\|_{\mathcal{B}(X)}\,r_0 \leq \delta_\omega \quad \text{ and } \quad \|\tilde W - \bar W\|_{\ell^1} \leq \|\pi_W A\|_{\mathcal{B}(X)}\,r_0 \leq \delta_W. \qedhere
\]
\end{proof}

\begin{remark}\label{rem:cost}
The size $N$ of the finite part is dictated by $Z_{1,2}$. Indeed, for each Fourier mode $n_1$, the norms of the columns of $\Lambda_{n_1}^{-1}$ have a sharp maximum at the turning point $n_2 \approx \frac12|n_1|\bar\omega$ of the recursion of Lemma~\ref{lem:rec}, and \eqref{eq:Z12product} is smaller than $1$ only if $\pi^{\leq N}$ contains all the columns whose norm exceeds $1/\|p\bar W^{p-1}\|_{\ell^1}$. With the values of Table~\ref{tab:bounds}, \eqref{eq:Z12product} is at most $0.85$, $0.86$ and $0.92$ for $(9,3)$, $(7,5)$ and $(6,7)$ respectively. The columns just outside $\pi^{\leq N}$ for which \eqref{eq:Z12product} exceeds $\vartheta = 0.8$ ($86$, $56$ and $1008$ columns) are computed using \eqref{eq:colMV}, and their quotients in Lemma~\ref{lem:Z1} are at most $0.048$, so that $Z_{1,2} \leq 0.8$. The cost of the proof is dominated by the finite part, and the proofs took $9.4$, $4.2$ and $14.5$ hours on $32$ cores. For the profiles of smaller dimensions mentioned in Remark~\ref{rem:intro}, the Chebyshev coefficients decay more slowly, which imposes $\nu < 1.05$. Then $\|p\bar W^{p-1}\|_{\ell^1}$ exceeds $10^6$, and $Z_{1,2} < 1$ would require a finite part of dimension $10^6$ or more.
\end{remark}

\begin{proof}[Proof of Proposition~\ref{prop:profiles}]
Let $\tilde U = (\tilde\omega, \tilde W)$ be given by Proposition~\ref{th:proof}, and let $\omega^* := \tilde\omega$ and
\[
W^*(\theta,\rho):=\sum_{n \in \Z\times\N_0}\tilde W_ne^{in_1\theta}T_{n_2}(2\rho^2-1).
\]
Since $\tilde W \in \ell^1$ with $\nu_\theta > 1$ and $\nu > 1$, the above series converges absolutely and locally uniformly, together with all its derivatives, for $|\mathrm{Im}(\theta)| < \log \nu_\theta$ and $2\rho^2-1 \in \Ell_\nu$. In particular, $W^*$ is real-analytic on $\R\times(-\rho_\nu,\rho_\nu)$, where $\rho_\nu > 1.0011=\rho_*$ is given in \eqref{eq:rhonu}, it is $2\pi$-periodic in $\theta$ and even in $\rho$. Moreover, $W^*$ is real-valued since $\tilde U$ is real, and $W^*(\theta+\pi,\rho)=-W^*(\theta,\rho)$ since $\tilde W_n = 0$ if $n_1$ is even. Now, differentiating the series term by term, $F(\tilde U) = 0$ implies that \eqref{eq:modes} is satisfied with $\omega = \omega^*$, that is, $W^*$ solves \eqref{eq:ssw}. Moreover, \eqref{eq:enclosure} yields $\omega^* > 0$ and, since $\|\bar W\|_{\ell^1} > \delta_W$ (see Tables~\ref{tab:proof} and~\ref{tab:bounds}), we have $\tilde W \neq 0$ and $W^*$ is not identically zero. Thus $W^*$ is a DSS profile of frequency $\omega^*$ in the sense of Definition~\ref{def:profile}, real-analytic on $\R\times(-\rho_*,\rho_*)$. The bound \eqref{eq:profilebound} is proved in Remark~\ref{rem:size} below. Finally, \eqref{eq:enclosure} and \eqref{eq:omega bar} provide \eqref{eq:omega}, since $\omega_{d,p}$ and the value of $\omega_0$ in Table~\ref{tab:proof} differ by less than $5\cdot10^{-11}$.
\end{proof}

\begin{remark}\label{rem:size}
Let $\bar W(\theta,\rho)$ be the function represented by $\bar W$. Since $|e^{in_1\theta}T_{n_2}(2\rho^2-1)| \leq 1$ for all $(\theta,\rho) \in \R\times[0,1]$, \eqref{eq:enclosure} yields $|W^*-\bar W| \leq \delta_W$ on $\R\times[0,1]$. Moreover, $|\bar W| \leq \sum_n|\bar W_n| =: S$ and $|\partial_\theta\bar W| \leq \sum_n|n_1||\bar W_n|$ on $\R\times[0,1]$. Evaluating $\bar W(\theta,0)$ and $\bar W(\theta,1)$ with interval arithmetic on a uniform grid of $2001$ points of $[0,\pi]$, and using the above bound on $\partial_\theta\bar W$ and that $\bar W(\theta+\pi,\rho) = -\bar W(\theta,\rho)$, we obtain in \cite{code} a lower bound $m$ for $\max\{|\bar W(\theta,0)|,|\bar W(\theta,1)|\}$, valid for all $\theta \in \R$. Consequently,
\[
  m-\delta_W \leq \max_{0 \leq \rho \leq 1}|W^*(\theta,\rho)| \leq S+\delta_W \quad \text{ for all } \theta \in \R.
\]
For $(d,p) = (9,3)$, $(7,5)$ and $(6,7)$, we obtain $m-\delta_W \geq 27.55$, $1.244$ and $1.008$, which proves \eqref{eq:profilebound}, and $S+\delta_W \leq 336.9$, $23.07$ and $11.26$, respectively. Thus, for the solution $u^*$ of \eqref{eq:ustar}, $\|u^*(t)\|_{L^\infty(|x|\le T-t)}=(T-t)^{-\alpha}\max_{0\le\rho\le1}|W^*(\theta,\rho)|$ is comparable to $(T-t)^{-\alpha}$ from above and below.
\end{remark}

\appendix
\section{Rough self-similar blow-up}\label{app:rough}

The solutions of Theorem~\ref{thm:main} are smooth, and their profiles are real-analytic across the light cone. If one gives up regularity at the light cone, radial finite-energy solutions of the defocusing equation whose supremum becomes infinite in finite time exist for every energy-supercritical exponent, and their construction is elementary. This is a known phenomenon (see \cite[footnote~2]{Ta} and the self-similar solutions in Besov spaces of \cite{Pl,RY}); we give here a self-contained version with bounded, compactly supported data, which shows precisely which regularity is lost, and where. As explained in Remark~\ref{rem:linear}, the same loss of boundedness occurs for the free wave equation by linear focusing.

In this appendix $d\ge3$ is an integer, $p>1+\frac4{d-2}$ is real, and $u^p$ stands for $|u|^{p-1}u$.    Lemma~\ref{lem:ss} and the bound \eqref{eq:radialbound}, which are used below, hold for real $p$ with the same proofs: the former only uses $p\alpha=\alpha+2$, and the latter only uses the gradient part of the energy. With $\alpha=\frac2{p-1}$ and $s_c=\frac d2-\alpha$ as in \eqref{eq:supercritical}, we write
\begin{equation}\label{eq:appconst}
\ell:=\frac{d-1}2,\qquad\beta:=\ell-\alpha=s_c-\frac12,\qquad c_0:=\alpha(\alpha+1),
\end{equation}
so that the energy-supercritical condition $s_c>1$ reads $\beta>\frac12$. The exponent $\beta$ is the singular Frobenius exponent $s_c-\frac12$ of Remark~\ref{rem:noCSS} for the mode $n=0$.

\begin{theorem}\label{thm:rough}
There exist real radial compactly supported data
\[
u_0\in H^1(\R^d)\cap L^\infty(\R^d),\qquad u_1\in L^2(\R^d)\cap L^\infty(\R^d),
\]
and a distributional solution $u$ of \eqref{eq:nlw} on $[0,1)\times\R^d$ with these data, such that, for every $S<1$,
\[
u\in C([0,S];H^1(\R^d))\cap C^1([0,S];L^2(\R^d)),\qquad u\in C([0,S]\times\R^d)\cap L^\infty([0,S]\times\R^d),
\]
$u$ is unique in this class, $u(t)$ is supported in $\{|x|\le2+t\}$, and the energy $E(u)(t)$ is finite and conserved. Moreover, there are $A>0$, $t_*\in(0,1)$ and a positive profile $f\in C([0,1])\cap C^\infty([0,1))$, smooth and even at $\rho=0$, such that
\begin{equation}\label{eq:latecss}
u(t,x)=(1-t)^{-\alpha}f\Big(\frac{|x|}{1-t}\Big)\qquad\text{for }t_*\le t<1,\ |x|<1-t,
\end{equation}
and $u(t,0)=A(1-t)^{-\alpha}$ for every $0\le t<1$. In particular $\|u(t)\|_{L^\infty(\R^d)}\ge A(1-t)^{-\alpha}\to\infty$ as $t\uparrow1$.
\end{theorem}

The solution is continuously self-similar in the late backward cone, so that no periodic modulation in $\tau$ is needed, and it is not smooth across its characteristic interfaces: the profile $f$ is a small perturbation of $A\psi$, which contains a nonsmooth contribution of order $(1-\rho)^\beta$, with a logarithmic factor when $\beta$ is an integer (Remark~\ref{rem:linear}); $f$ is not $C^1$ there when $\beta\le1$ (Lemma~\ref{lem:fprime}), and the data are not Lipschitz in that case. We do not rule out a weak continuation past $t=1$. The construction has two steps: a small positive $\tau$-independent profile, obtained by a contraction in $C([0,1])$ from the regular solution of the mode-$0$ linear equation, and a symmetric characteristic problem which removes an unbounded part of the initial velocity near $|x|=1$; the latter is needed exactly when $\beta\le1$.

\subsection{A small positive self-similar profile}
By Lemma~\ref{lem:ss}, $U(t,r)=(1-t)^{-\alpha}f(r/(1-t))$ solves \eqref{eq:nlw} in $\{r<1-t\}$ if and only if
\begin{equation}\label{eq:css}
Lf=f^p,\qquad Lf:=(1-\rho^2)f''+\Big(\frac{d-1}\rho-2(\alpha+1)\rho\Big)f'-c_0f .
\end{equation}
We look for a solution with $f(0)=A>0$ and $f'(0)=0$. With $P(\rho):=\rho^{d-1}(1-\rho^2)^{1-\beta}$ one checks that $Lw=\frac{1-\rho^2}{P}(Pw')'-c_0w$. The regular solution of $L\psi=0$ with $\psi(0)=1$ is the hypergeometric function of Lemma~\ref{lem:hyp} for $n=0$,
\[
\psi(\rho)={}_2F_1\Big(\frac\alpha2,\frac{\alpha+1}2;\frac d2;\rho^2\Big),
\]
whose series has positive coefficients; by Gauss's summation formula \cite[(15.4.20)]{DLMF}, since $\frac d2-\frac\alpha2-\frac{\alpha+1}2=\beta>0$,
\begin{equation}\label{eq:H}
\psi(1)=H:=\frac{\Gamma(\frac d2)\Gamma(\beta)}{\Gamma(\frac{d-\alpha}2)\Gamma(\frac{d-\alpha-1}2)}<\infty,\qquad H>1 .
\end{equation}

\begin{lemma}\label{lem:green}
The operator
\begin{equation}\label{eq:green}
(\mathcal Gg)(\rho):=\psi(\rho)\int_0^\rho\frac1{P(z)\psi(z)^2}\int_0^z\frac{P(t)\psi(t)}{1-t^2}\,g(t)\,dt\,dz
\end{equation}
is a positive bounded operator on $C([0,1])$, $L\mathcal Gg=g$ with $(\mathcal Gg)(0)=(\mathcal Gg)'(0)=0$, and $\|\mathcal G\|=K:=\frac{H-1}{c_0}$.
\end{lemma}

\begin{proof}
The axis singularity is integrable: the inner integral is $O(z^d)$ and the outer integrand is $O(z)$. The identity $L(\psi v)=\frac{1-\rho^2}{P\psi}(P\psi^2v')'$ gives the equation and the conditions at $0$. Uniqueness of the regular solution of the linear problem gives $\mathcal G1=\frac{\psi-1}{c_0}$, since $L1=-c_0$. The kernel is nonnegative and increases with $\rho$, including the factor $\psi(\rho)$, hence for $\rho_2\ge\rho_1$
\[
|(\mathcal Gg)(\rho_2)-(\mathcal Gg)(\rho_1)|\le\|g\|_\infty\frac{\psi(\rho_2)-\psi(\rho_1)}{c_0},
\]
which proves the continuity of $\mathcal Gg$ up to $\rho=1$, by \eqref{eq:H}, and $\|\mathcal G\|=\|\mathcal G1\|_\infty=K$.
\end{proof}

Let
\begin{equation}\label{eq:smallness}
M:=(4pK)^{-\frac1{p-1}},\qquad A:=\frac M{2H},
\end{equation}
and consider the map $f\mapsto A\psi+\mathcal G(f^p)$ on the order interval $\{0\le f\le M\}$ of $C([0,1])$. It takes values between $A\psi>0$ and $AH+KM^p=\frac M2+\frac M{4p}<M$, and its Lipschitz constant is $pKM^{p-1}=\frac14$. Its unique fixed point $f$ solves \eqref{eq:css}, with $A\le f\le M$, and $f$ is increasing, as the sum of two increasing functions. For $\rho<1$ the function $f$ is smooth, including at the axis: in the Volterra form of \eqref{eq:css} (the integral formula for $f'$ in the proof of Lemma~\ref{lem:fprime} below), which gives $f'=O(\rho)$, successive differentiations in $\rho^2$ bootstrap every derivative, and since $f\ge A>0$ the power $f^p$ is smooth for every real $p$. We fix
\begin{equation}\label{eq:U}
U(t,r):=(1-t)^{-\alpha}f\Big(\frac r{1-t}\Big),\qquad0\le r<1-t ,
\end{equation}
a classical solution of \eqref{eq:nlw} in the open cone.

\begin{lemma}\label{lem:fprime}
$f\in H^1(B_1)$, and as $\rho\uparrow1$,
\[
f'(\rho)=\begin{cases}O\big((1-\rho)^{\beta-1}\big)&\text{if }\tfrac12<\beta<1,\\ O(1+|\log(1-\rho)|)&\text{if }\beta=1,\\ O(1)&\text{if }\beta>1.\end{cases}
\]
\end{lemma}

\begin{proof}
Integrating $(Pf')'=\frac P{1-\rho^2}(c_0f+f^p)$,
\[
f'(\rho)=\rho^{1-d}(1-\rho^2)^{\beta-1}\int_0^\rho t^{d-1}(1-t^2)^{-\beta}\big(c_0f(t)+f(t)^p\big)\,dt,
\]
where the bracket is positive and bounded; the integral of $(1-t)^{-\beta}$ gives the three cases. Near the axis $f'=O(\rho)$, and $f'$ is square integrable near $\rho=1$ since $2\beta-2>-1$.
\end{proof}

The initial velocity of $U$ is $\partial_tU(0,r)=\alpha f(r)+rf'(r)$: it is unbounded near $r=1$ when $\beta\le1$. This is why the whole slice $t=0$ of $U$ is not used as the initial datum.

\subsection{A symmetric characteristic problem}
For radial functions let $v:=r^\ell u$ and $x:=r+t$, $y:=r-t$, so that $r=\frac{x+y}2$. Then \eqref{eq:nlw} reads
\begin{equation}\label{eq:vxy}
v_{xy}=F(x,y,v),\qquad F(x,y,z):=\frac{\ell(\ell-1)}{(x+y)^2}z+\frac14\Big(\frac{x+y}2\Big)^{-\ell(p-1)}z^p,
\end{equation}
and $F(x,y,z)=F(y,x,z)$. With $C_0:=\ell(\ell-1)$, $C_1:=2^{\ell(p-1)-2}$ and $L_*:=C_0+pC_1(4M)^{p-1}$, we fix
\begin{equation}\label{eq:R}
\eta:=\min\Big\{\frac14,\frac1{2\sqrt{1+L_*}}\Big\},\qquad R:=1-\eta,\qquad t_*:=\frac\eta2 ,
\end{equation}
so that $R\ge\frac34$ and $\eta^2L_*<\frac14$. On the characteristic $y=R$ we prescribe the trace of $U$,
\begin{equation}\label{eq:g}
g(x):=\Big(\frac{x+R}2\Big)^\ell\,U\Big(\frac{x-R}2,\frac{x+R}2\Big),\qquad R\le x\le1,
\end{equation}
extended by continuity to $x=1$, where $r=1-t$, using $f(1)$. Along this segment $r\le s:=1-t\le1$, hence $r^\ell s^{-\alpha}\le s^\beta\le1$: $g$ is positive, increasing, $\|g\|_\infty\le M$, and since $s^2-r^2=(1+R)(1-x)$, Lemma~\ref{lem:fprime} gives
\begin{equation}\label{eq:gH1}
g\in C([R,1])\cap H^1(R,1).
\end{equation}

\begin{lemma}\label{lem:goursat}
The integral equation
\begin{equation}\label{eq:goursat}
V(x,y)=g(x)+g(y)-g(R)+\int_R^x\!\!\int_R^yF(\xi,\zeta,V(\xi,\zeta))\,d\zeta\,d\xi,\qquad(x,y)\in[R,1]^2,
\end{equation}
has a unique solution with $\|V\|_\infty\le4M$. It is continuous, positive and symmetric, $V(x,y)=V(y,x)$, its weak derivatives
\begin{equation}\label{eq:Vx}
V_x(x,y)=g'(x)+\int_R^yF(x,\zeta,V(x,\zeta))\,d\zeta,\qquad V_y(x,y)=g'(y)+\int_R^xF(\xi,y,V(\xi,y))\,d\xi
\end{equation}
are square integrable, $V$ is a classical solution of \eqref{eq:vxy} away from $\{x=1\}\cup\{y=1\}$, and $(\partial_x-\partial_y)V(r,r)=0$ for $R\le r<1$.
\end{lemma}

\begin{proof}
On the square $r=\frac{x+y}2\ge R\ge\frac34$, so that $|F(x,y,z)|\le C_0|z|+C_1|z|^p$ and $|F(x,y,z)-F(x,y,w)|\le L_*|z-w|$ for $|z|,|w|\le4M$. The boundary term of \eqref{eq:goursat} has norm at most $2M$ and the integral term at most $4M\eta^2L_*<M$ on the ball $\|V\|_\infty\le4M$, on which the map is a contraction of constant $\eta^2L_*<\frac14$. Uniqueness and the symmetry of \eqref{eq:vxy} give $V(x,y)=V(y,x)$; iterating from the boundary term, which is at least $g(R)>0$, and using $F(x,y,z)\ge0$ for $z\ge0$, gives $V>0$. The formulas \eqref{eq:Vx} follow from \eqref{eq:goursat}, and $g'\in L^2$ by \eqref{eq:gH1}; differentiating them gives the classical regularity away from $x=1$ and $y=1$. On the diagonal, the symmetry gives $(\partial_x-\partial_y)V(r,r)=0$.
\end{proof}

On the diagonal $x=y=r$, which is the initial slice $t=0$, $\partial_t=\partial_x-\partial_y$: the solution of Lemma~\ref{lem:goursat} has zero initial velocity on $R\le r<1$, and its initial displacement $V(r,r)$ is bounded with the weak derivative $2g'(r)$ plus a bounded function, hence belongs to $H^1(R,1)$.

\subsection{Gluing in the backward cone and proof of Theorem~\ref{thm:rough}}
On $0\le t<1$, $0\le r<1-t$, define
\begin{equation}\label{eq:patch}
Q(t,r):=\begin{cases}U(t,r),&r-t<R,\\ r^{-\ell}V(r+t,r-t),&r-t\ge R,\end{cases}
\end{equation}
where the second case lies in the square of Lemma~\ref{lem:goursat}, since $R\le r-t\le r+t<1$. On the interface $y=r-t=R$ the two pieces have, in the variable $v=r^\ell u$, the same trace $g(x)$, by \eqref{eq:g} and \eqref{eq:goursat}, so that their tangential derivatives $v_x=g'(x)$ agree too, and $v_{xy}=\partial_y(v_x)$ has no singular part on the interface: $Q$ solves \eqref{eq:vxy}, hence \eqref{eq:nlw}, in the distributional sense in the open cone (the normal derivative may jump across the characteristic $y=R$, as it may across any characteristic). The axis is in the region of $U$. On every cone $r+t\le1-\eps$ the two pieces have bounded derivatives on compact time intervals, their gluing is in the energy class, and \eqref{eq:Vx}, \eqref{eq:gH1} give a finite energy on the slices up to the endpoint $r+t=1$, continuous in $t$ (continuity of translations in $L^2$, and the width of the patch vanishes at $t=t_*$). Finally, if $t\ge t_*=\frac{1-R}2$ and $r<1-t$ then $r-t<1-2t\le R$: for $t\ge t_*$, $Q=U$ in the whole backward cone, and $Q(t,0)=U(t,0)=A(1-t)^{-\alpha}$ for every $t$.

\begin{proof}[Proof of Theorem~\ref{thm:rough}]
\emph{Data.} For $r<1$ let
\[
u_0(r):=\begin{cases}f(r),&r<R,\\ r^{-\ell}V(r,r),&R\le r<1,\end{cases}\qquad u_1(r):=\begin{cases}\alpha f(r)+rf'(r),&r<R,\\0,&R<r<1,\end{cases}
\]
which are the data of $Q$. The displacement is continuous at $r=R$, since $V(R,R)=g(R)=R^\ell f(R)$, and $u_0\in H^1(B_1)\cap L^\infty(B_1)$ by the above; $u_1$ is bounded, since $R<1$, with a jump at $r=R$. Let $b:=u_0(1)$ and let $\chi$ be a smooth radial cutoff, equal to $1$ on $[0,\frac54]$ and to $0$ on $[2,\infty)$; for $r\ge1$ let $u_0(r):=b\chi(r)$ and $u_1(r):=0$. The displacement is continuous at $r=1$, so its weak gradient has no surface measure, and $(u_0,u_1)$ are in the spaces of the theorem, with a finite energy $E_0$.

\emph{Cauchy realization.} For $J>0$ let $\mathcal N_J(z):=\mathrm{sgn}(z)\min\{|z|^p,J^p\}$ and $\Phi_J(z):=\int_0^z\mathcal N_J$, so that $0\le\Phi_J(z)\le\frac{|z|^{p+1}}{p+1}$. The nonlinearity $\mathcal N_J$ is globally Lipschitz, and the linear energy estimates with Duhamel's formula give a unique global energy solution $u_J$ of $\partial_t^2u-\Delta u+\mathcal N_J(u)=0$ with the data $(u_0,u_1)$; its energy $\int\frac12u_t^2+\frac12|\nabla u|^2+\Phi_J(u)$ is conserved and bounded by $E_0$, and $u_J(t)$ is radial and supported in $\{r\le2+t\}$. Fix $0<\delta<1$. On $[0,1-\delta]$ the field $Q$ is bounded by $C_\delta:=\max\{M\delta^{-\alpha},4MR^{-\ell}\}$ in the cone, and for $J>C_\delta$ it solves the truncated equation there, with the same data as $u_J$ for $r<1$. The difference $w:=u_J-Q$ is then an energy solution of $w_{tt}-\Delta w+aw=0$ with $|a|\le pJ^{p-1}$ and zero data in $B_1$; the local energy $\mathcal E_b(t):=\frac12\int_{|x|<b-t}(w_t^2+|\nabla w|^2+w^2)\,dx$, $b<1$, satisfies $\frac d{dt}\mathcal E_b\le(1+pJ^{p-1})\mathcal E_b$, the boundary term having a sign on the characteristic cone, so that $w=0$ on every cone $r+t<b$, hence
\begin{equation}\label{eq:identify}
u_J=Q\qquad\text{for }r<1-t,\ 0\le t\le1-\delta .
\end{equation}
This comparison uses the whole field \eqref{eq:patch} and the annular data: finite speed of propagation from the unchanged core alone would not give an expanding self-similar region. Inside the ball $r<\frac\delta2$ we have $u_J=U$ by \eqref{eq:identify}, and outside it the radial bound \eqref{eq:radialbound}, which uses only the gradient part of the energy, gives $|u_J(t,x)|\le C_Er^{-\frac{d-2}2}$ with $C_E:=(2E_0/((d-2)|\mathbb S^{d-1}|))^{1/2}$, uniformly in $J$. Choosing $J>\max\{C_\delta,C_E(\delta/2)^{-\frac{d-2}2}\}$, the truncation is inactive on $[0,1-\delta]$, and $u_J$ solves \eqref{eq:nlw} there. For different large $J$ these solutions agree on their common intervals, by the same difference estimate, and define a solution $u$ on $[0,1)$ with the asserted energy continuity and conservation (on each interval the truncated potential equals $\frac{|u|^{p+1}}{p+1}$); the difference estimate also gives the uniqueness in the class of energy solutions with bounded displacement.

\emph{Continuity and blow-up.} Radial $H^1$ functions have continuous representatives on annuli away from the origin, so that $u\in C_tH^1$ is jointly continuous there, and $u=U$ near the axis: $u$ is continuous on $[0,S]\times\R^d$ for every $S<1$. Finally \eqref{eq:identify} gives \eqref{eq:latecss}, $u(t,0)=A(1-t)^{-\alpha}$, and the blow-up of the supremum.
\end{proof}

\begin{remark}[The linear mechanism]\label{rem:linear}
The proof covers every integer $d\ge3$ and every real $p>1+\frac4{d-2}$ with the same two contractions, without any numerical step; the supercriticality enters only through $\beta>\frac12$, which makes the characteristic trace $g$ an $H^1$ function. The displacement has an unbounded gradient at $r=1$ when $\beta\le1$, and bounded jumps of derivatives across $r=R$ and across the characteristic interfaces, so that no smooth-data statement follows, in any dimension.

The same loss of boundedness already occurs for the free wave equation. Indeed, since $L\psi=0$, the computation in the proof of Lemma~\ref{lem:ss} shows that
\begin{equation}\label{eq:free}
v(t,x):=(1-t)^{-\alpha}\,\psi\Big(\frac{|x|}{1-t}\Big)
\end{equation}
solves the \emph{free} wave equation $\partial_t^2v-\Delta v=0$ in $\K_1$. It is bounded by $H(1-t)^{-\alpha}$, it satisfies $v(t,0)=(1-t)^{-\alpha}$, and its data on the unit ball, $(\psi,\alpha\psi+\rho\psi')$, belong to $H^1\times L^2$ by the proof of Lemma~\ref{lem:fprime}, and to $H^s\times H^{s-1}$ for every $s<s_c$, but not for $s=s_c$: by the connection formula \cite[(15.8.4)]{DLMF}, $\psi$ contains the singular branch $(1-\rho^2)^\beta$ at $\rho=1$ with a nonzero coefficient, namely $\Gamma(\frac d2)\Gamma(-\beta)/\big(\Gamma(\frac\alpha2)\Gamma(\frac{\alpha+1}2)\big)$ when $\beta\notin\N$ (when $\beta\in\N$ the branch carries a logarithm). The unboundedness of $v$ at the vertex of the cone is the focusing of the conormal singularity of order $\beta$ carried by the sphere $\{|x|=1\}$ at time $0$. 

The displayed data have bounded displacement, but their velocity is unbounded when $\beta\le1$. Applying the symmetric characteristic construction above with $f=A\psi$ and with the nonlinear term removed gives a free wave with both initial fields bounded, compactly supported and of finite energy, and with the same growth $A(1-t)^{-\alpha}$ on the axis. These are analogous linear data; the comparison does not identify the free evolution of the nonlinear solution's Cauchy data.

For the nonlinear solution, the identity $f=A\psi+\mathcal G(f^p)$ gives, in the late cone,
\[
u(t,x)=Av(t,x)+(1-t)^{-\alpha}(\mathcal G(f^p))\Big(\frac{|x|}{1-t}\Big).
\]
The correction vanishes on the axis, so that $u(t,0)=Av(t,0)=A(1-t)^{-\alpha}$. Along every fixed similarity ray $|x|=\rho(1-t)$ with $0<\rho<1$, the positive Green kernel gives $\mathcal G(f^p)(\rho)>0$, so the correction also grows like $(1-t)^{-\alpha}$. Thus loss of boundedness at this rate alone does not distinguish nonlinear blow-up from linear focusing. The supercriticality enters through the finite energy of the data.

By contrast, the profiles of Theorem~\ref{thm:main} are real-analytic across the light cone, and the free wave equation has no nonzero DSS profile with this property, for any frequency $\omega>0$. The zero mode is excluded by the preceding expansion of $\psi$. For $n\ne0$, the solution ${}_2F_1(a_n,b_n;\frac d2;\rho^2)$ of the mode equation $L_nV=0$ of Lemma~\ref{lem:hyp} which is regular at the axis contains the singular branch $(1-\rho^2)^{s_c-\frac12-in\omega}$ (the branch $(1-\rho)^{s_c-\frac12-in\omega}$ of Remark~\ref{rem:noCSS}), with the nonzero coefficient $\Gamma(\frac d2)\Gamma(-\beta+in\omega)/\big(\Gamma(a_n)\Gamma(b_n)\big)$, by the same connection formula. In this sense Theorem~\ref{thm:rough} locates the difficulty of the smooth problem: not the existence of a self-similar singularity of the defocusing equation, but its regularity across the light cone.
\end{remark}

\end{document}